\documentclass{amsart}     % Basic GT/GTM/AGT style
\usepackage{graphicx}  %%% the recommended graphics package
\usepackage[percent]{overpic}
\usepackage{enumerate}
\usepackage{enumitem}
\usepackage{thm-restate}
\usepackage{hyperref}
\usepackage{cleveref}
\usepackage{todonotes}
\usepackage{lscape}
\usepackage{multirow}
\usepackage{MnSymbol}
\usepackage{amsmath}

\newcommand{\rmv}{\smallsetminus}
\newcommand{\bd}{\partial}
\newcommand{\set}[1]{ \left\{ #1 \right\} }

\newcommand{\io}{\iota}

\title{Genus two Goeritz equivalence in lens spaces $L(p, 1)$}

\author{Brandy Doleshal}
\address{Brandy Doleshal
\newline Sam Houston State University
\newline Huntsville, TX
\newline USA}
\email{bdoleshal@shsu.edu}
\urladdr{}

\author{Matt Rathbun}
\address{Matt Rathbun
\newline California State University, Fullerton
\newline Fullerton, CA
\newline USA}
\email{mrathbun@fullerton.edu}
\urladdr{}
\newtheorem{thm}{Theorem}[section]    % Standard theorem environment
\newtheorem{lem}[thm]{Lemma}          % Lemma environment with numbering 
\theoremstyle{definition}
\newtheorem{rem}{Remark}             % Unnumbered environment for remarks.

\newtheorem{claim}[thm]{Claim}

\begin{document}

\maketitle

\begin{abstract}
In this paper, we consider the action of the Goeritz group $\mathcal G_p$ for the genus two Heegaard splitting of the lens space $L(p,1)$ with $p\ge 2$ on the homology of the Heegaard surface. We describe the action in terms of matrices in $GL(4, \mathbb Z)$, and provide homology and homotopy obstructions for when two curves in the Heegaard surface are  Goeritz equivalent. 
\end{abstract}

\section{Introduction}

The Berge conjecture has been a motivating goal for low-dimensional topology for the last 50 years, using special positions of a knot on a genus 2 Heegaard splitting and the induced surface slopes to characterize the kinds of manifolds that can arise from Dehn surgery on the knot. This open conjecture has driven a great deal of research to answer the question (e.g., \cite{GreLSRP}, \cite{HedFHBCKALSS}, \cite{NiKFHDFK}), as well as variations and generalizations (e.g., \cite{BakHofLicJPKSBLS}, \cite{DeaSSFDSHK}, \cite{LiuLSSL}).

In exploring some of these variations, some surprising results have been found, including \cite{AmoDolRatACPTTK}, \cite{EudOHKSFDS}, and \cite{GunKDPPPSR}, which find infinite families of knots, each with inequivalent positions on a genus 2 Heegaard splitting in $S^3$ but with equal induced surface slopes and corresponding surgeries giving rise to the same non-hyperbolic manifolds. The latter two authors of \cite{AmoDolRatACPTTK} endeavored to build a larger framework for simplifying the question of when two positions of a knot on a genus 2 Heegaard splitting are inequivalent, via the action of the Goeritz group. 

The \textit{Goeritz group} of a Heegaard splitting of a $3$-manifold is the group of isotopy classes of orientation-preserving automorphisms of the manifold preserving the Heegaard splitting, set-wise. In \cite{DolRatG2GES3}, we provide homological obstructions for two curves on the (unique) genus 2 Heegaard surface in $S^3$ to be related by Goeritz group elements. The current work continues this enterprise. Building on work of Cho \cite{ChoG2GGLS}, which provides a presentation for each of the Goeritz groups for the (unique) genus 2 Heegaard splitting of lens spaces, $L(p, 1)$, we are able to emulate the methods of \cite{DolRatG2GES3}, and provide homology obstructions in the same spirit, albeit more byzantine in some cases.

In Section \ref{section:GoeritzGroupLp1}, we present background on the Goeritz group for $L(p,1)$ and the induced action on the homology of a genus 2 Heegaard surface. In Section \ref{section:Star}, we explore the images and kernels of these induced actions, and in Section \ref{section:Obstructions} we provide obstructions to the existence of such actions. We prove Theorem \ref{thm:linalgobstruction}, which provides a collection of relatively straightforward necessary arithmetic conditions for a Goeritz action between two curves in terms of their homology vectors. Theorem \ref{thm:vectorlist} spells out seven explicit relationships between the homology vectors' curves, at least one of which must be satisfied in order for the curves to be Goeritz equivalent. Finally, in Theorem \ref{thm:HomotopyObstruction}, we provide an obstruction to two curves being related by a Goeritz element in terms of their fundamental group representatives in each handlebody.

\section{The Goeritz group of $L(p, 1)$}
\label{section:GoeritzGroupLp1}

Observe that every Heegaard splitting of a lens space is standard: a stabilization of the unique genus 1 Heegaard splitting. We will exclusively examine genus 2 Heegaard splittings, so we will often simply refer to \emph{the} Goeritz group of the lens space $L(p,1)$, for $p \geq 2$. Cho uses the complex of primitive disks to provide a presentation for each of the Goeritz groups of $L(p,1)$.

\begin{thm}[Cho \cite{ChoG2GGLS}]\label{thm:Cho}
For $p \ge 2$, a presentation for the genus-2 Goeritz group $\mathcal G_p$ for the lens space $L(p,1)$ is given by:
\begin{enumerate}
\item $\langle \, \beta, \, \rho, \, \gamma \ \mid \ \rho^4 = \gamma^2 = (\gamma \rho)^2 = \rho^2 \beta \rho^2 \beta^{-1} = 1 \, \rangle$ if $p = 2$, 
\item $\langle \, \alpha \ \mid \ \alpha^2 = 1 \, \rangle \bigoplus \langle \, \beta, \, \delta, \, \gamma \ \mid \ \delta^3 = \gamma^2 = (\gamma \delta)^2 = 1 \, \rangle$ if $p = 3$, and
\item $\langle \, \alpha \ \mid \ \alpha^2 = 1 \, \rangle \bigoplus \langle \, \beta, \, \gamma, \, \sigma \ \mid \ \gamma^2 = \sigma^2 = 1 \, \rangle$ if $p \ge 4$.
\end{enumerate}
\end{thm}
In Theorem \ref{thm:Cho}, the maps $\beta$, $\rho$, $\gamma$, $\alpha$, $\delta$, and $\sigma$ refer to specific homeomorphisms. Let $(F', V', W')$ be the genus one Heegaard splitting of $L(p, 1)$, with the $(p, 1)$-torus knot bounding a disk in $V'$, and let $(F, V, W)$ be the genus 2 stabilization of $(F', V', W')$. We consider the action $\mathcal G_p$ on $H_1(F; \mathbb Z)$ by considering how the elements affect a basis for homology. As most of the homeomorphisms in \cite{ChoG2GGLS} are described in terms of the handlebody $W$, we will choose a standard basis for the homology of $F$ from this perspective. Let the homology vectors $\textbf{a}, \textbf{x}, \textbf{b}$, and $\textbf{y}$ be the standard basis vectors, represented by the oriented curves $a$, $x$, $b$, and $y$ shown in Figure \ref{figure:alphabeta}, with $a$ and $x$ bounding disks in $W$, $y$ bounding a disk in $V$, and the cross products $a \times b$ and $x \times y$ (of the tangent vectors) pointing into $V$ at the points of intersection. The actions then produce a map $\star: \mathcal{G}_p \to GL(4, \mathbb Z)$. We denote by $\alpha_*, \beta_*, \gamma_*, \delta_*, \rho_*$, and $\sigma_*$ the automorphisms of $H_1(F; \mathbb{Z})$ induced by $\alpha, \beta, \gamma, \delta, \rho$, and $\sigma$, respectively. 

\begin{lem}
\label{lemma:matrices}
With respect to the ordered basis $\set{\bf{a}, \bf{x}, \bf{b}, \bf{y}}$, the matrices induced by the generators of $\mathcal{G}_p$ are as follows.
\begin{gather*}
\alpha_* = \begin{pmatrix} -1 & 0 & 0 & 0 \\ 0 & -1 & 0 & 0 \\ 0 & 0 & -1 & 0 \\ 0 & 0 & 0 & -1\end{pmatrix},\  \beta_*  = \begin{pmatrix} 1 & 0 & 0 & 0 \\ 0 & -1 & 0 & 0 \\ 0 & 0 & 1 & 0 \\ 0 & 0 & 0 & -1\end{pmatrix}, \ \gamma_*  = \begin{pmatrix} 1 & 1 & 0 & 0 \\ 0 & -1 & 0 & 0 \\ 0 & 0 & 1 & 0 \\ 0 & 0 & 1 & -1\end{pmatrix},   \\ 
\delta_*   =  \begin{pmatrix} -2 & -1 & 1 & -1 \\ 3 & 1 & -1 & 0 \\ 0 & 0 & 1 & -3 \\ 0 & 0 & 1 & -2\end{pmatrix}, \rho_*   =  \begin{pmatrix} 1 & 1 & 1 & -1 \\ -2 & -1 & -1 & 0 \\ 0 & 0 & -1 & 2 \\ 0 & 0 & -1 & 1\end{pmatrix},\mbox{ and } \sigma_* = \begin{pmatrix} 1 & 0 & 0 & 1 \\ -p & -1 & -1 & 0 \\ 0 & 0 & 1 & -p \\ 0 & 0 & 0 & -1 \end{pmatrix}.
\end{gather*}
\end{lem}

\begin{figure}
\begin{center}
\begin{overpic}[scale=.6]{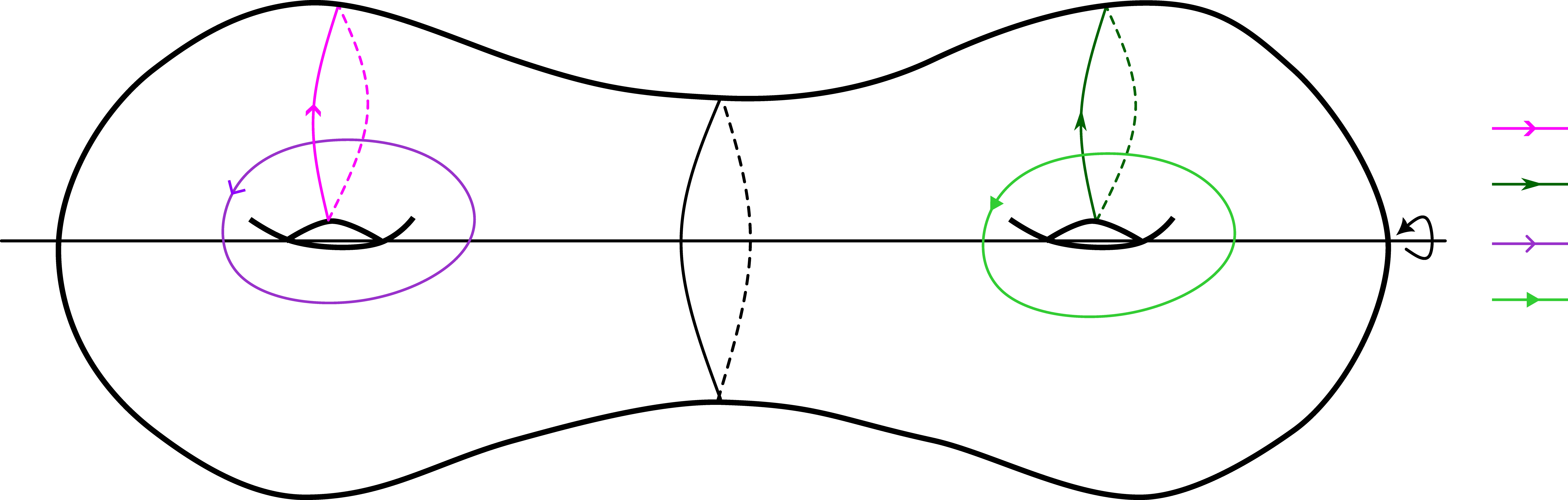}
\put (50,30) {$V$}
\put (50,17.5) {$W$}
\put (101.5,22.8) {$a$}
\put (101.5,19.3) {$x$}
\put (101.5,15.2) {$b$}
\put (101.5,11.7) {$y$}
\end{overpic}
\caption{Standard homology basis for $F$, and the axis of rotation for the homeomorphisms $\alpha$ and $\beta$.}
\label{figure:alphabeta}
\end{center}
\end{figure}
\begin{figure}
\begin{center}
\begin{overpic}[scale =.75]{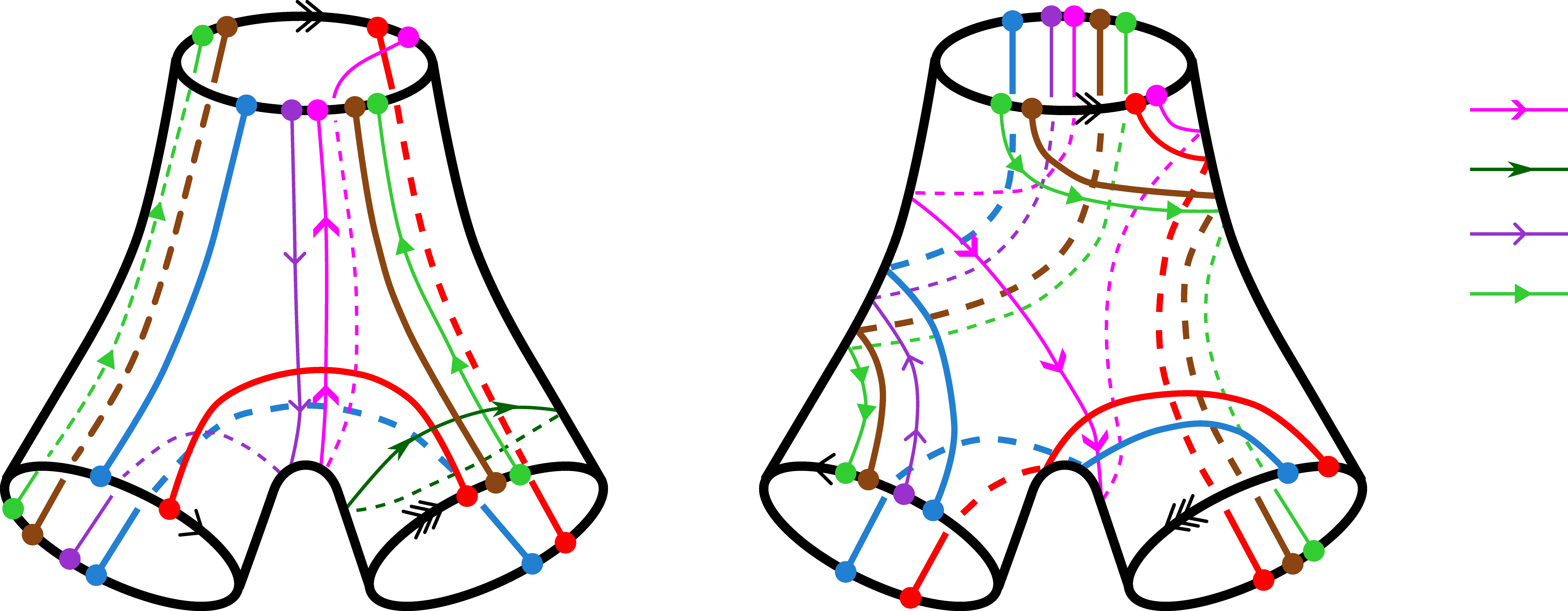}
\put (101,31.1) {$a$}
\put (101,27.3) {$x$}
\put (101,22.9) {$b$}
\put (101,19.3) {$y$}
\end{overpic}
\caption{Pants decomposition of $F$ used to describe $\delta$, together with the resulting curves and sub-arcs of the homology basis, and $\bd D$, $\bd E$, $\bd F$.}
\label{figure:deltabasepants}
\end{center}
\end{figure}
\begin{figure}
\begin{center}
\begin{overpic}[scale =.75]{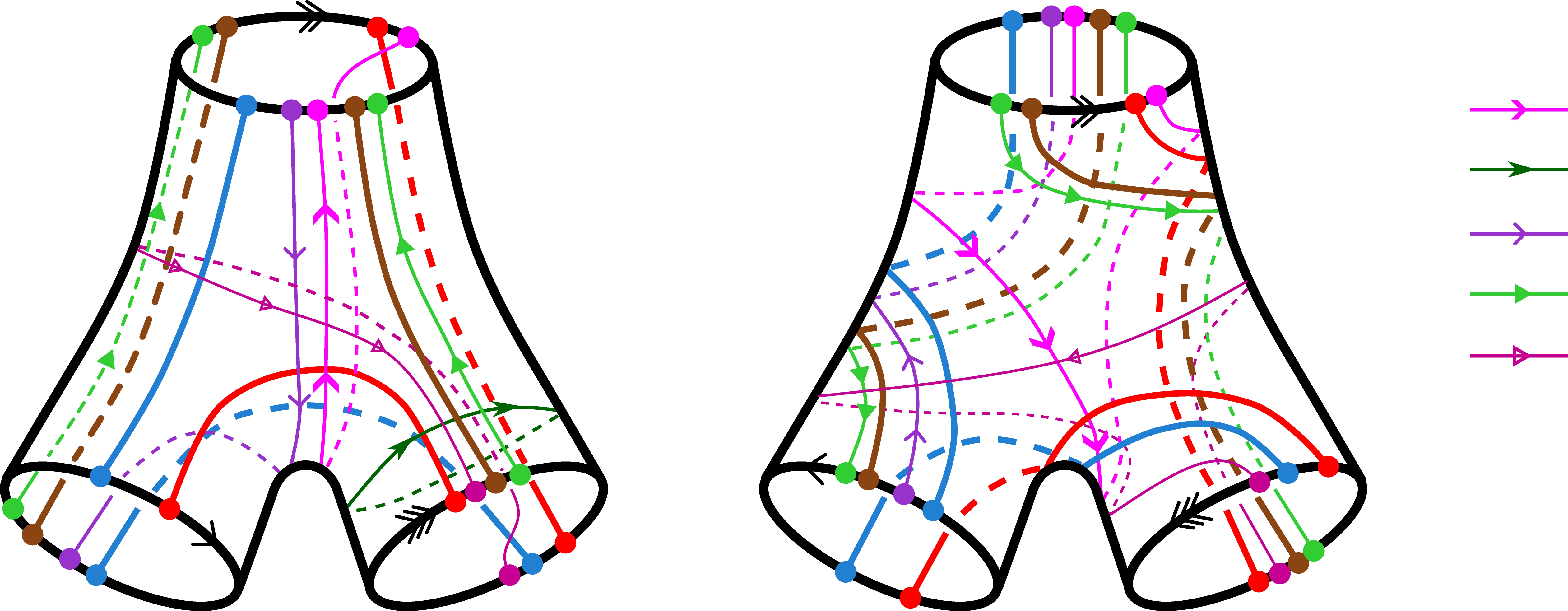}
\put (101,31.1) {$a$}
\put (101,27.3) {$x$}
\put (101,22.8) {$b$}
\put (101,19.4) {$y$}
\put (101, 15.1) {$\delta(a)$}
\end{overpic}
\caption{The image $\delta(a)$, superimposed onto Figure \ref{figure:deltabasepants}.}
\label{figure:deltaofa}
\end{center}
\end{figure}
\begin{figure}
\begin{center}
\begin{overpic}[scale =.75]{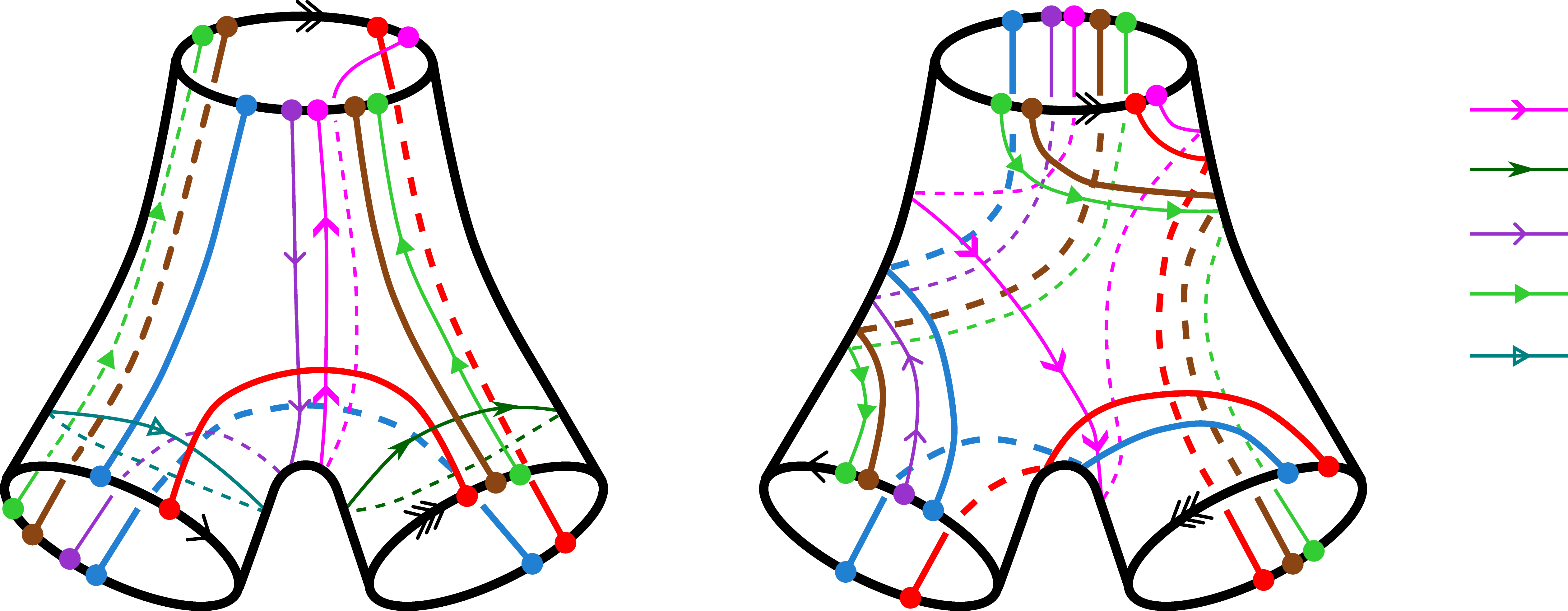}
\put (101,31.1) {$a$}
\put (101,27.3) {$x$}
\put (101,22.8) {$b$}
\put (101,19.3) {$y$}
\put (101, 15.1) {$\delta(x)$}
\end{overpic}
\caption{The image $\delta(x)$, superimposed onto Figure \ref{figure:deltabasepants}.}
\label{figure:deltaofx}
\end{center}
\end{figure}
\begin{figure}
\begin{center}
\begin{overpic}[scale =.75]{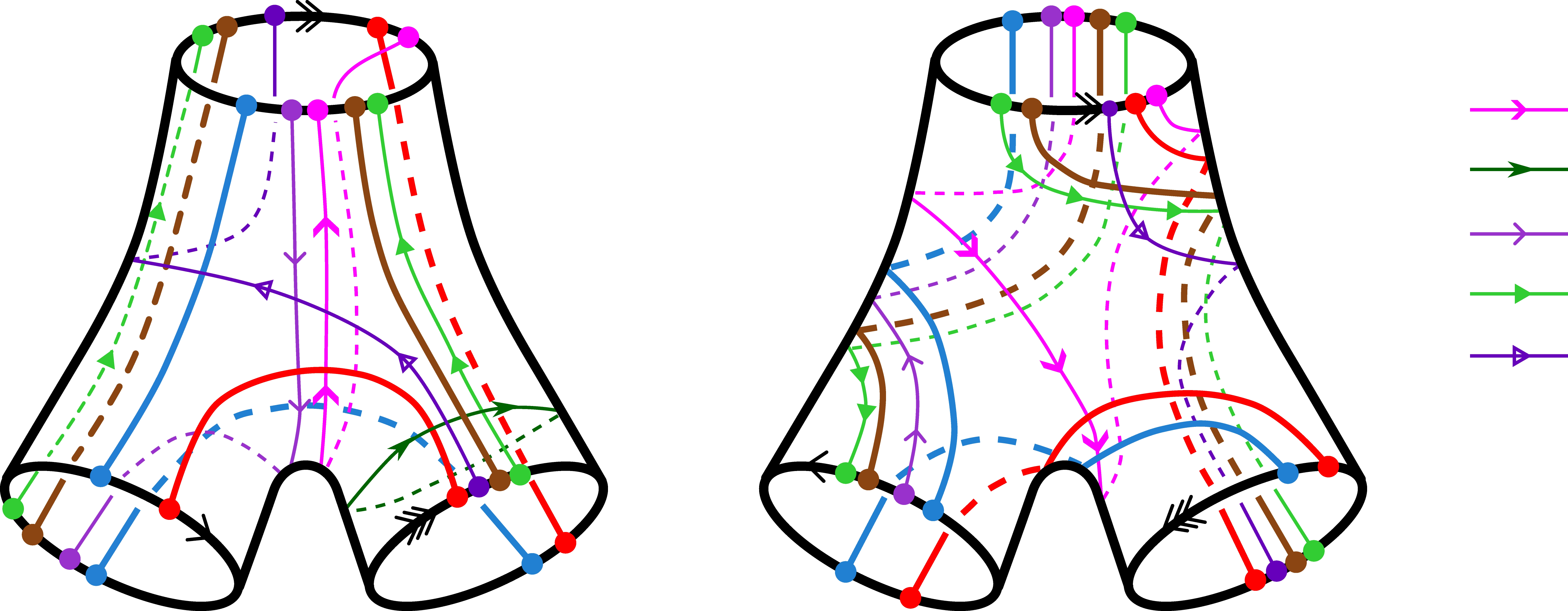}
\put (101,31.1) {$a$}
\put (101,27.4) {$x$}
\put (101,22.8) {$b$}
\put (101,19.3) {$y$}
\put (101, 15.1) {$\delta(b)$}
\end{overpic}
\caption{The image $\delta(b)$, superimposed onto Figure \ref{figure:deltabasepants}.}
\label{figure:deltaofb}
\end{center}
\end{figure}
\begin{figure}
\begin{center}
\begin{overpic}[scale =.75]{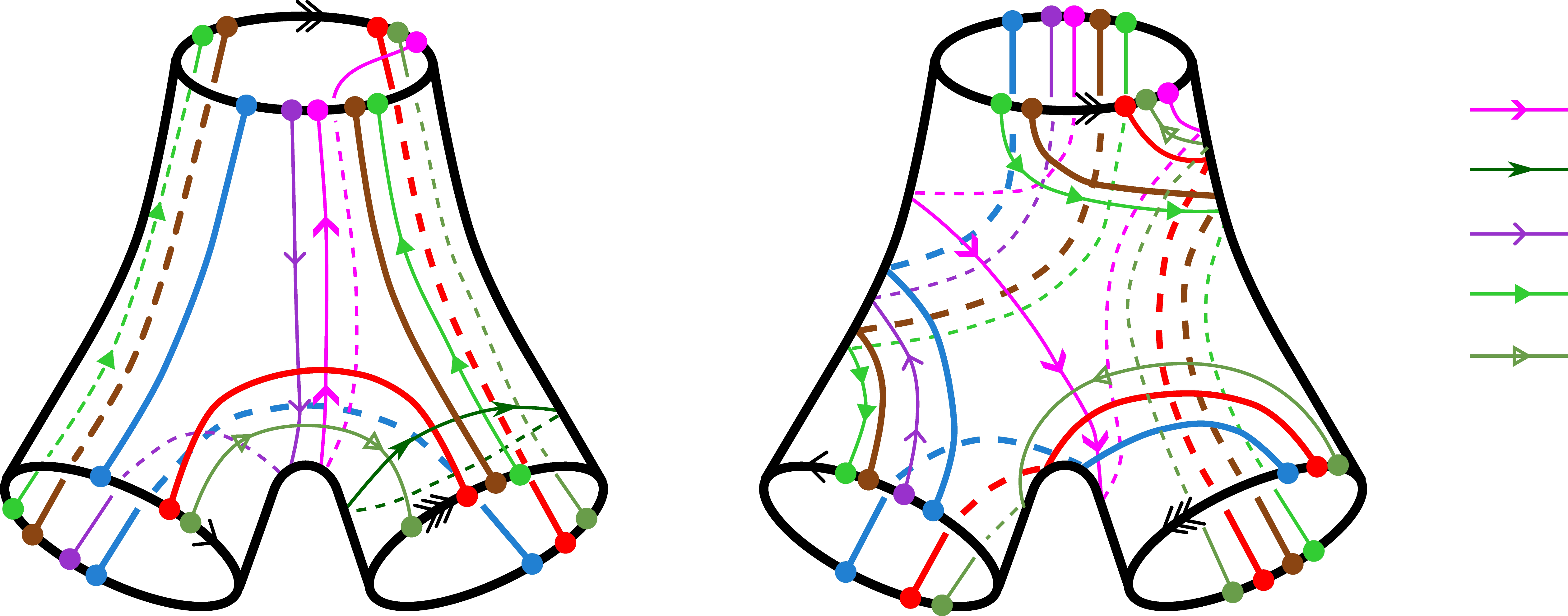}
\put (101,31.1) {$a$}
\put (101,27.4) {$x$}
\put (101,22.8) {$b$}
\put (101,19.5) {$y$}
\put (101, 15.1) {$\delta(y)$}
\end{overpic}
\caption{The image $\delta(y)$, superimposed onto Figure \ref{figure:deltabasepants}.}
\label{figure:deltaofy}
\end{center}
\end{figure}
\begin{figure}
\begin{center}
\begin{overpic}[scale =.7]{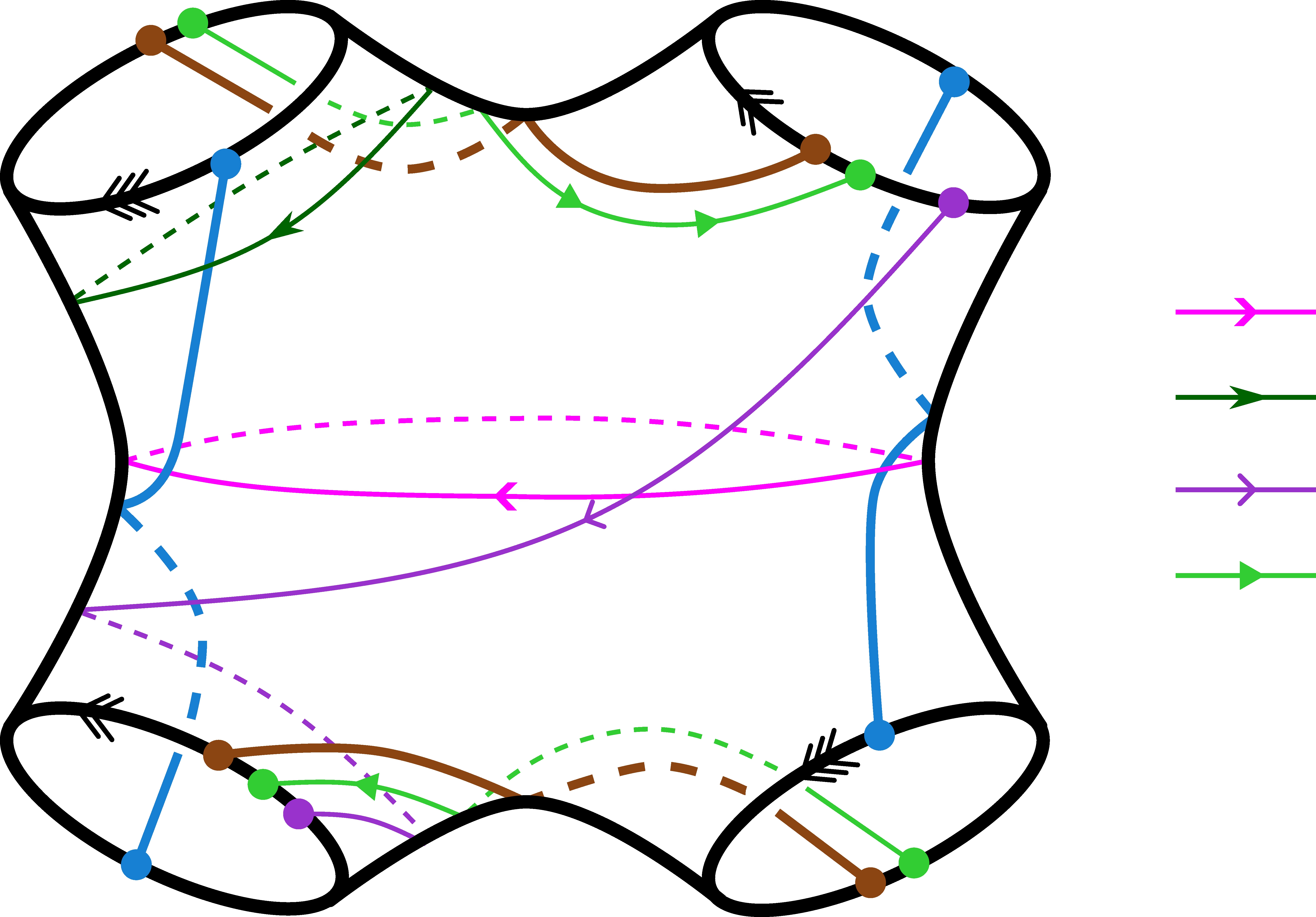}
\put (102,44.3) {$a$}
\put (102,37.6) {$x$}
\put (102,30.3) {$b$}
\put (102,24.3) {$y$}
\end{overpic}
\caption{Decomposition of $F$ used to describe $\rho$, together with the resulting curves and sub-arcs of the homology basis, and $\bd D$, $\bd E$, $\bd F$.}
\label{figure:rhobasepillowcase}
\end{center}
\end{figure}
\begin{figure}
\begin{center}
\begin{minipage}{0.5\textwidth}
\begin{overpic}[scale =.7]{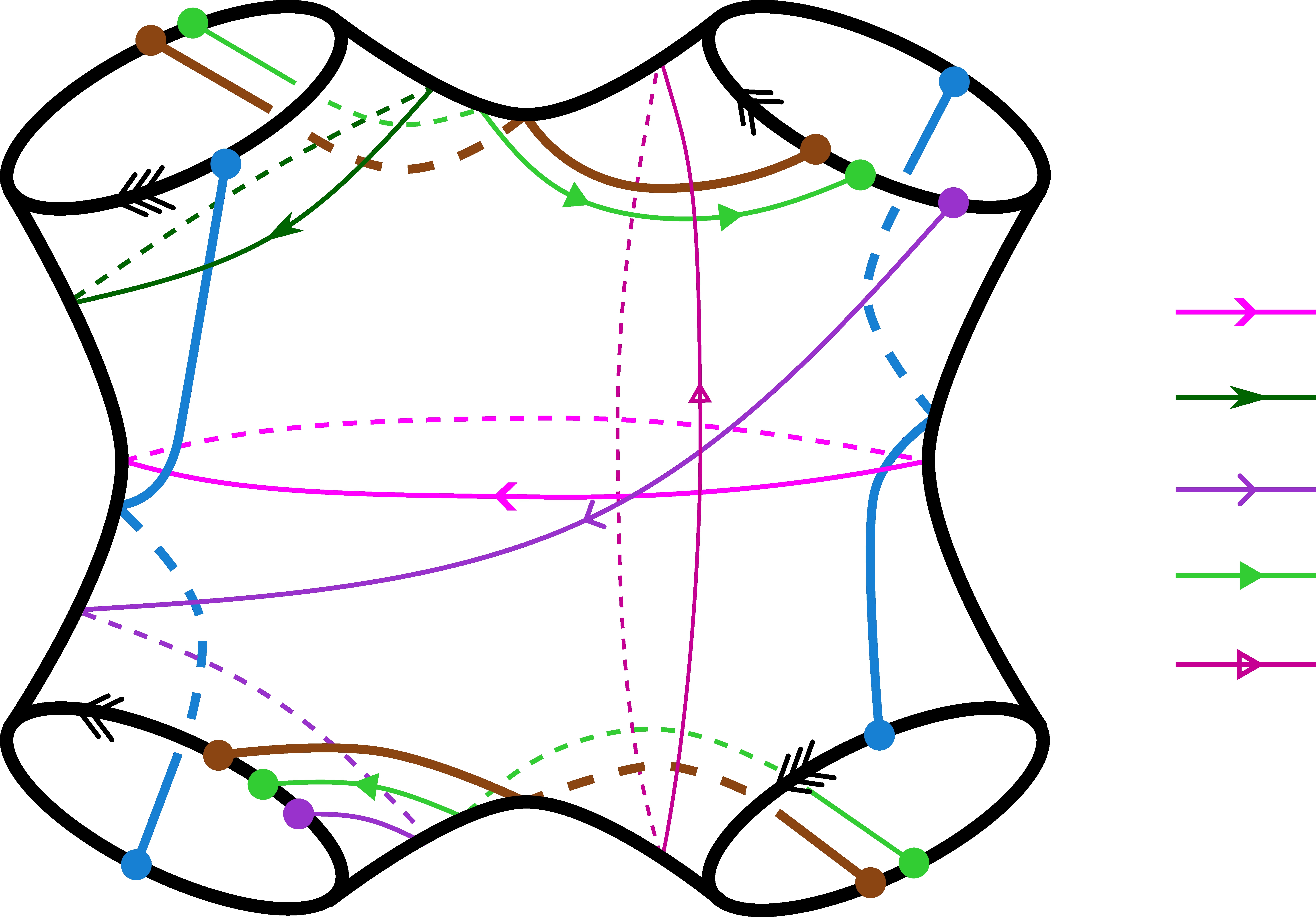}
\put (102,44.3) {$a$}
\put (102,37.6) {$x$}
\put (102,30.3) {$b$}
\put (102,24.3) {$y$}
\put (102,17.8) {$\rho(a)$}
\end{overpic}
\caption{The image $\rho(a)$ superimposed onto Figure \ref{figure:rhobasepillowcase}.}
\label{figure:rhoofa}
\end{minipage} \begin{minipage}{0.45\textwidth}
\begin{overpic}[scale =.7]{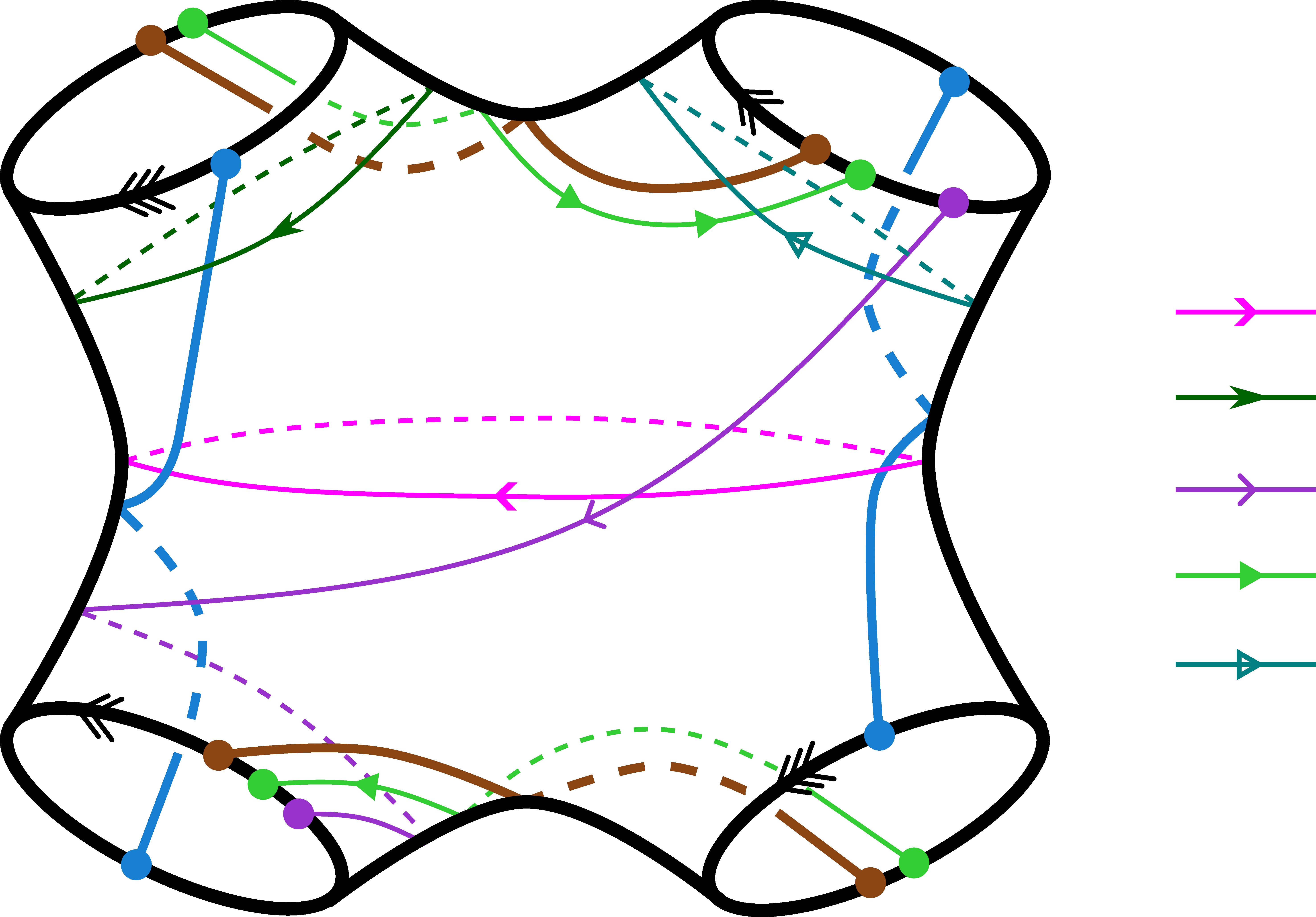}
\put (102,44.3) {$a$}
\put (102,37.6) {$x$}
\put (102,30.4) {$b$}
\put (102,24.3) {$y$}
\put (102,17.8) {$\rho(x)$}
\end{overpic}
\caption{The image $\rho(x)$ superimposed onto Figure \ref{figure:rhobasepillowcase}.}
\label{figure:rhoofx} \end{minipage}\vspace{.3in}
\begin{minipage}{0.5\textwidth}
\begin{overpic}[scale = .7]{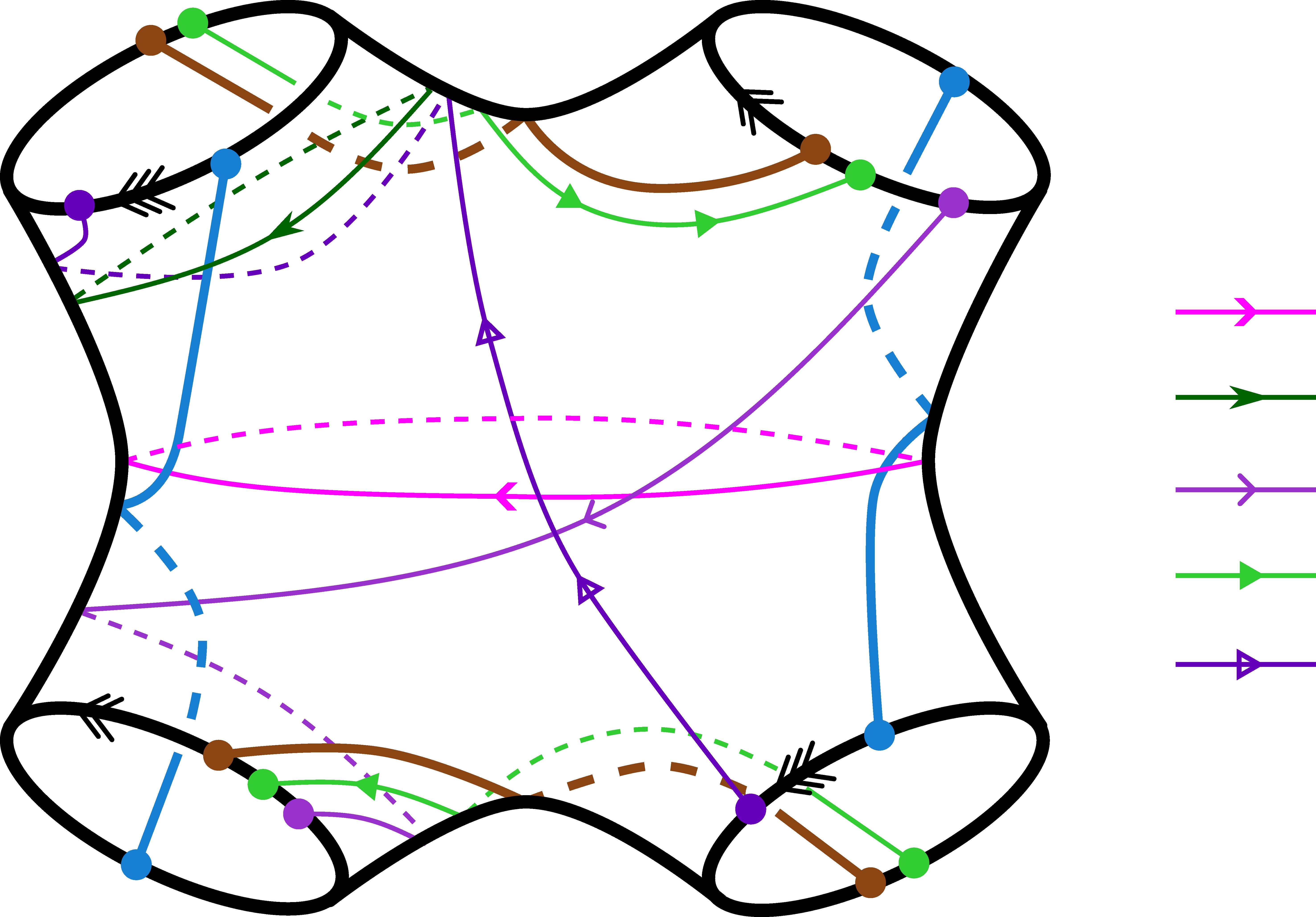}
\put (102,44.3) {$a$}
\put (102,37.6) {$x$}
\put (102,30.4) {$b$}
\put (102,24.3) {$y$}
\put (102,17.8) {$\rho(b)$}
\end{overpic}
\caption{The image $\rho(b)$ superimposed onto Figure \ref{figure:rhobasepillowcase}.}
\label{figure:rhoofb}
\end{minipage} \begin{minipage}{0.45\textwidth}
\begin{overpic}[scale =.7]{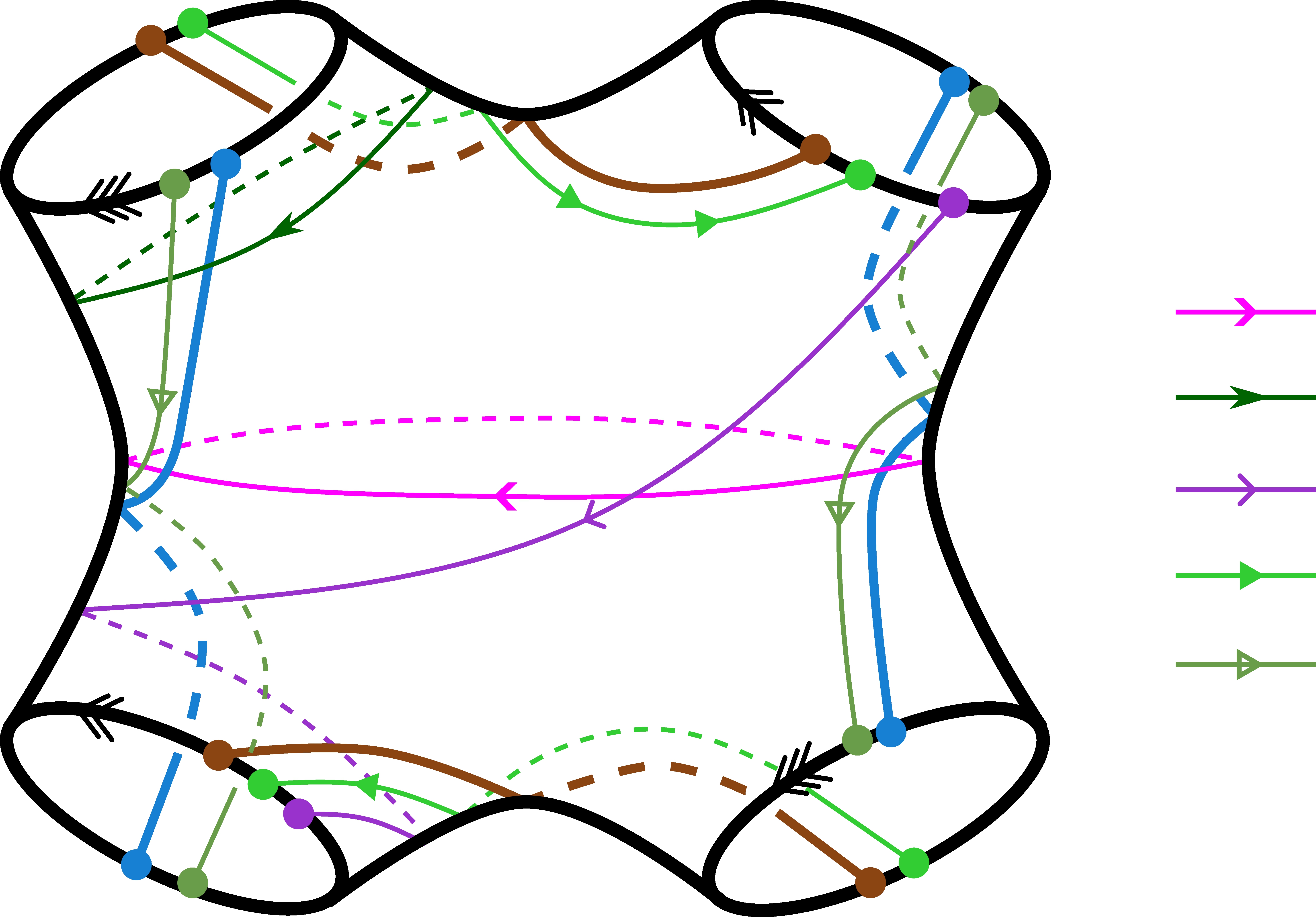}
\put (102,44.3) {$a$}
\put (102,37.6) {$x$}
\put (102,30.4) {$b$}
\put (102,24.3) {$y$}
\put (102,17.8) {$\rho(y)$}\end{overpic}
\caption{The image $\rho(y)$ superimposed onto Figure \ref{figure:rhobasepillowcase}.}
\label{figure:rhoofy} \end{minipage}
\end{center}
\end{figure}
\begin{figure}
\begin{center}
\begin{minipage}{.5\textwidth}
\begin{overpic}[scale =.5]{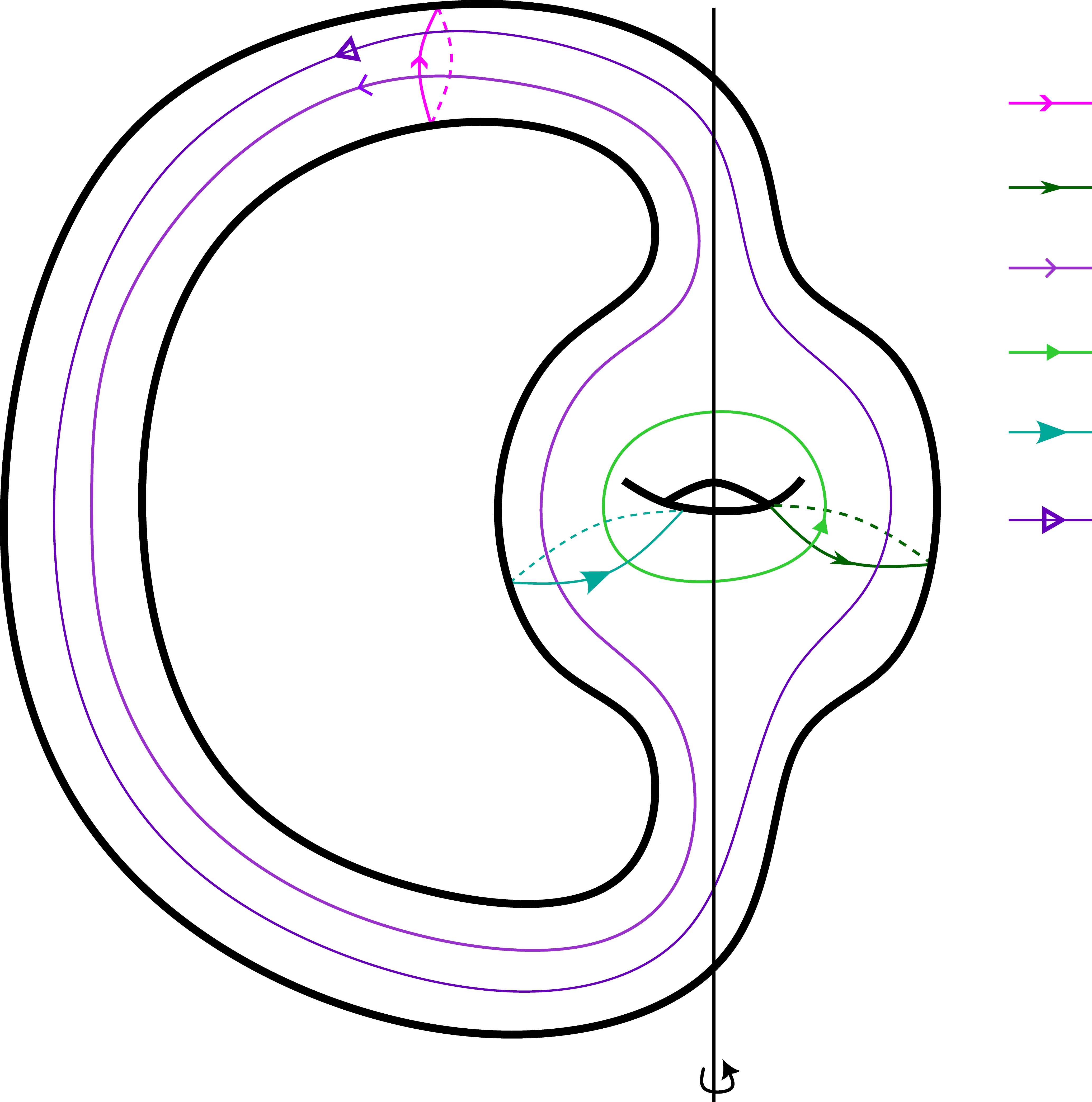}
\put (102,89.4) {$a$}
\put (102,81.5) {$x$}
\put (102,73.7) {$b$}
\put (102,66.5) {$y$}
\put (101.,59.) {$\gamma(x)$}
\put (101.,50.7) {$\gamma(b)$}
\put (5, 65.6) {$W$}
\put (25, 65.6) {$V$}
\end{overpic}
\caption{Axis of rotation for $\gamma$, and its action on the homology basis.}
\label{figure:gamma} \end{minipage}
\begin{minipage}{.45\textwidth}
\begin{overpic}[scale =.5]{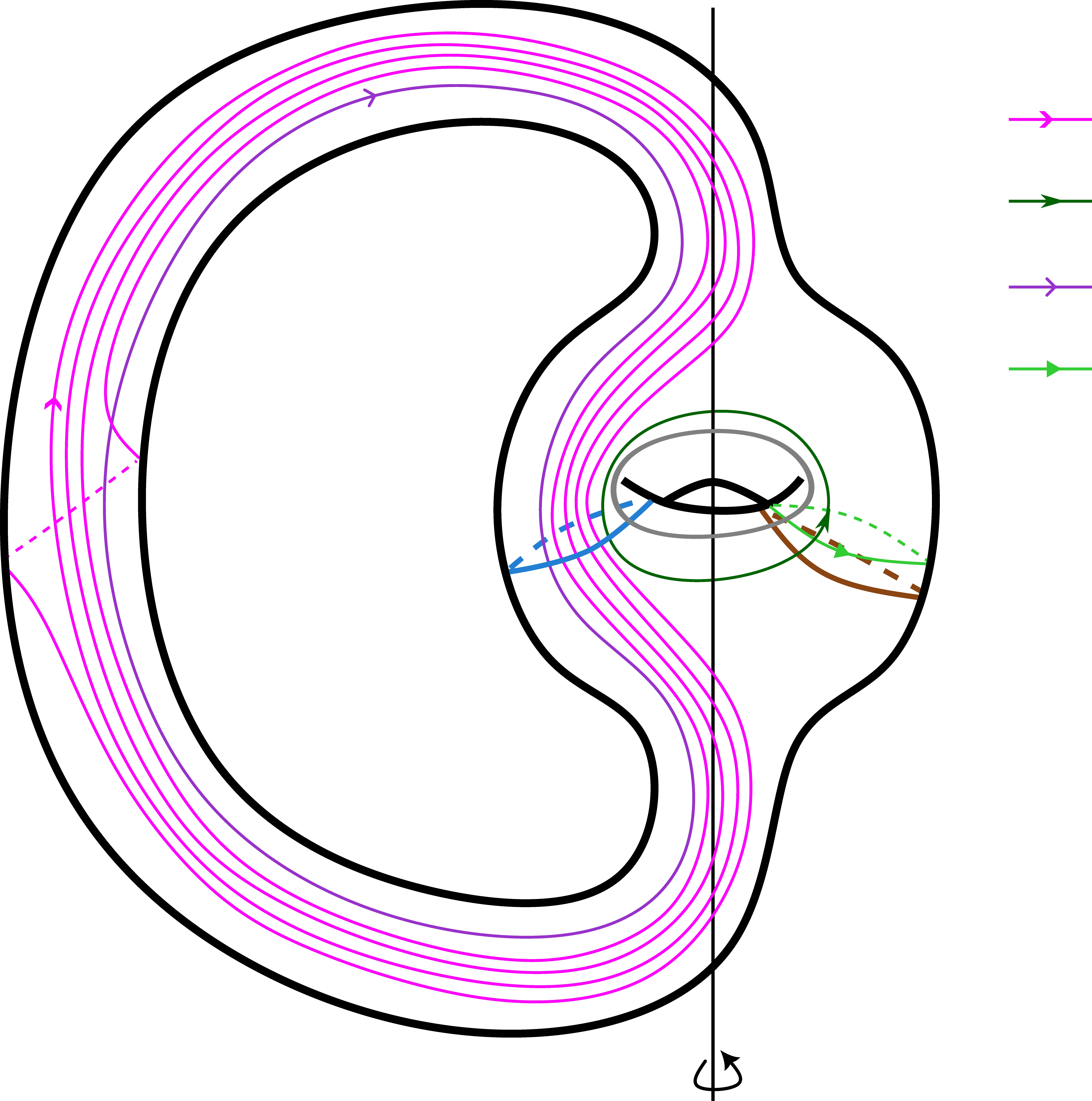}
\put (102,88.) {$a$}
\put (102,80) {$x$}
\put (102,72.) {$b$}
\put (102,64.9) {$y$}
\put (5, 65.6) {$V$}
\put (25, 65.6) {$W$}
\end{overpic}
\caption{The axis of rotation for $\sigma$, the homology basis, and $\bd D$, $\bd E$, $\bd E'$ (gray).}
\label{figure:sigma}\end{minipage}
\end{center}
\end{figure}
\begin{figure}
\begin{center}
\begin{overpic}[scale=.55]{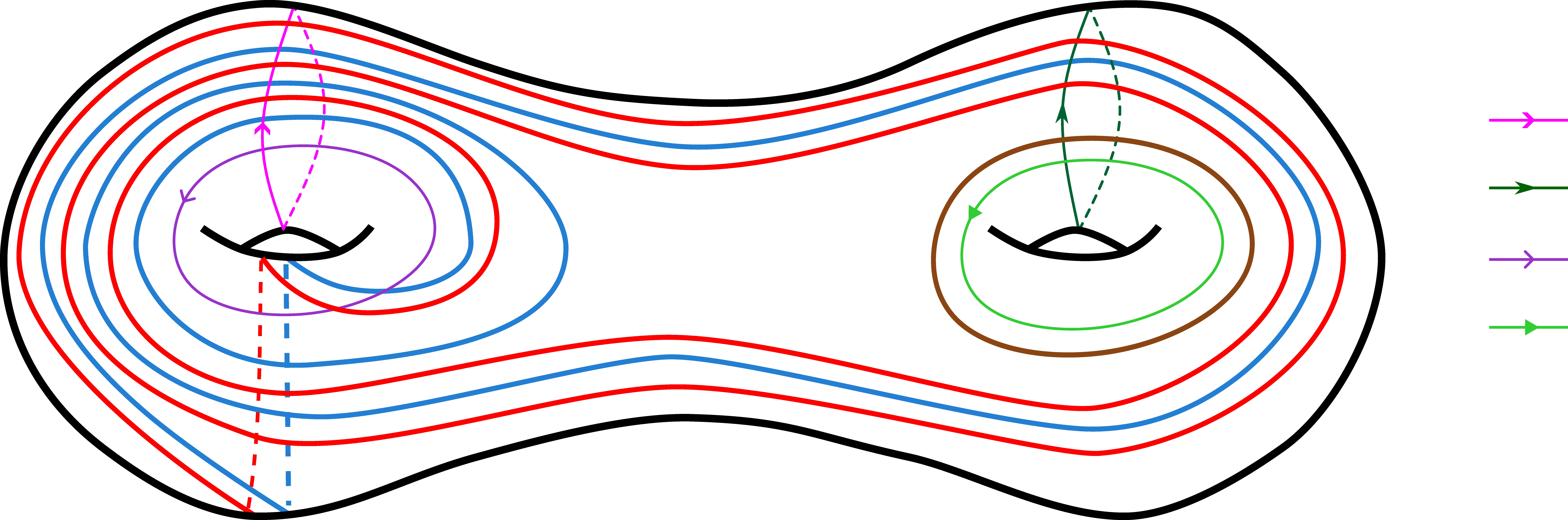}
\put (50,30) {$V$}
\put (48,14.75) {$W$}
\put (43.75,18.75) {$F$}
\put (55.75,17.75) {$E$}
\put (36.5,15.75) {$D$}
\put (102,24.5) {$a$}
\put (102,20.2) {$x$}
\put (102,15.3) {$b$}
\put (102,11.3) {$y$}
\end{overpic}
\caption{The generators $a, x, b,$ and $y$ and primitive disks $D, E,$ and $F$.}
\label{figure:disksandgens}
\end{center}
\end{figure}
\begin{figure}
\begin{center}
\begin{overpic}[scale =.55]{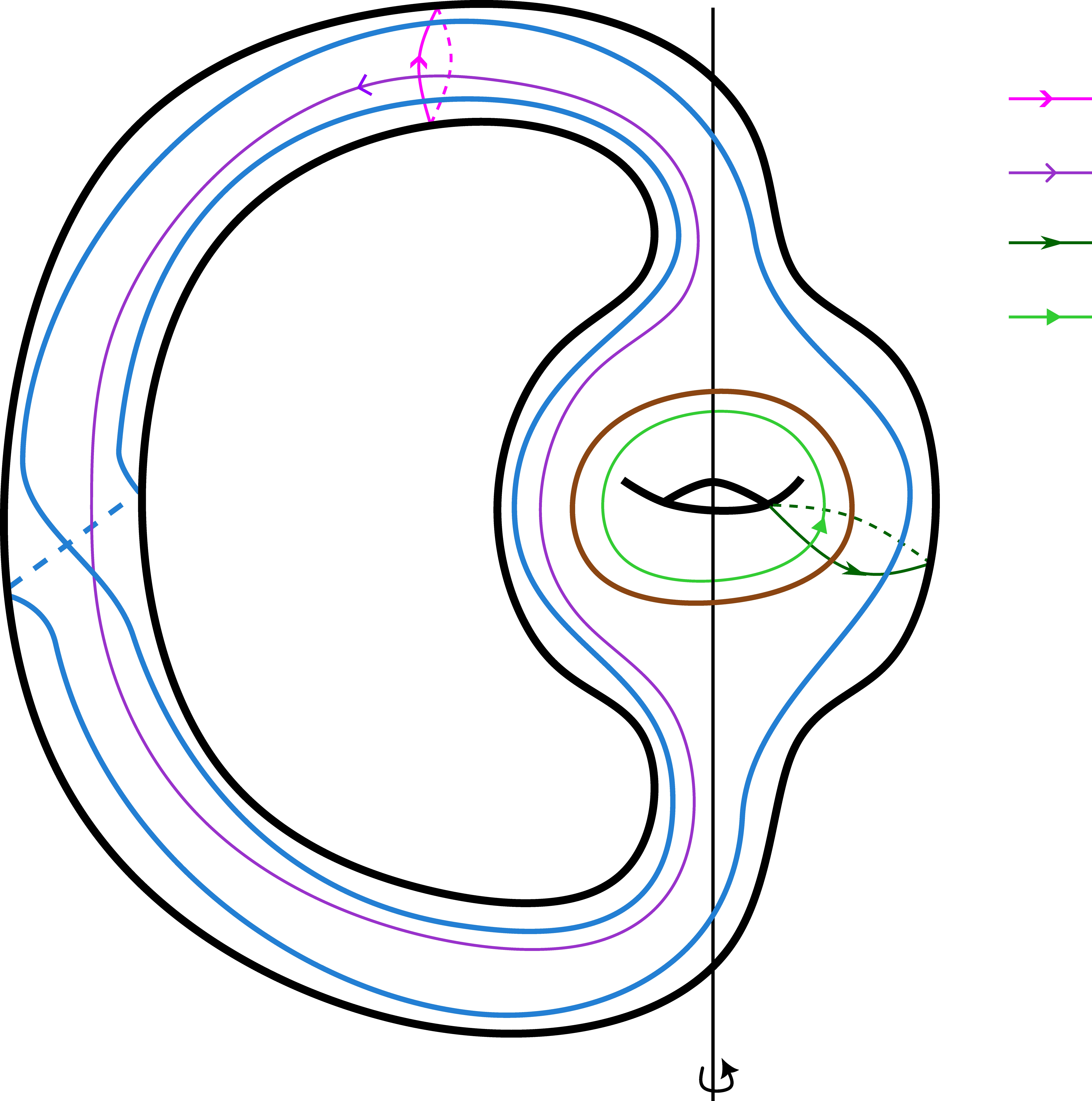}
\put (64.75,39.5) {$E$}
\put (43.5,8.25) {$D$}
\put (102,89.5) {$a$}
\put (102,82.8) {$x$}
\put (102,76.) {$b$}
\put (102,69.5) {$y$}
\put (5, 65.6) {$W$}
\put (25, 65.6) {$V$}
\end{overpic}
\caption{The generators $a, x, b,$ and $y$ and boundaries of primitive disks $D, E,$ and $F$ when $p = 2$.}
\label{figure:chofig8withgens}
\end{center}
\end{figure}

\begin{proof}
The proof will consist of tracking the images of the basis curves under each of the homeomorphisms in the generating set for $\mathcal{G}_p$. Let $\io(\cdot, \cdot)$ denote the oriented intersection number between two curves, with the positive surface orientation pointing into $V$. Observe, then, that for a curve $z$ representing vector $\textbf{z}$, and an element $g \in \mathcal{G}_p$, 
\begin{eqnarray} \label{eqn:vector} g_*(\textbf{z}) = (\io(g(z), b), \, \io(g(z), y), \, \io(a, g(z)), \, \io(x, g(z)))^T. \end{eqnarray} 

First, the maps $\alpha$ and $\beta$ are relatively straightforward. The map $\alpha$ is a rotation by $\pi$ about the horizontal axis of symmetry of $F$, as shown in Figure \ref{figure:alphabeta}. It is the hyperelliptic involution; it simply reverses the orientation of each curve. The map $\beta$ is a rotation of half of $F$ by $\pi$ about the same axis, with the right handle of the surface rotating clockwise, as viewed from the right in the figure. This map reverses the orientations of $x$ and $y$, leaving $a$ and $b$ fixed. These produce the matrices shown for $\alpha_*$ and $\beta_*$.

Next, the map $\gamma$ is shown in Figure \ref{figure:gamma}, rotating the right handle by $\pi$ about the vertical axis indicated. The curve $a$ is fixed, $y$ is reversed, and the images of $b$ and $x$ are shown. Equation (\ref{eqn:vector}) then provides the matrix shown for $\gamma_*$.

When $p=2$, there is a 4-fold rotation, $\rho$, on the Heegaard surface. There exists a primitive pair of disks $D'$ and $E'$ in $W$, and a unique primitive pair of disks $D$ and $E$ in $V$ that are commonly dual to $D'$ and $E'$. Cutting $W$ along $D'$ and $E'$ results in a 3-ball, with four distinguished disks on the boundary sphere, which we can think of as a 4-holed sphere. The curves $\bd D$ and $\bd E$ on this sphere are symmetrically arranged (see Figure (8) from \cite{ChoG2GGLS}) to give rise to a 4-fold rotation. This 4-holed sphere, together with the curves $a$, $x$, $b$, and $y$, is shown in Figure \ref{figure:rhobasepillowcase}. The map $\rho$ rotates the 4-holed sphere clockwise by $\pi/2$, exchanging $D$ and $E$, and exchanging $D'$ and $E'$. We can then analyze the intersections of the images $\rho(a)$, $\rho(x)$, $\rho(b)$, and $\rho(y)$, respectively shown in Figures \ref{figure:rhoofa}, \ref{figure:rhoofx}, \ref{figure:rhoofb}, and \ref{figure:rhoofy}, with all four of the original curves $a$, $x$, $b$, and $y$. The columns of $\rho_*$ are then produced from Equation \ref{eqn:vector}.

When $p = 3$, there is a map $\delta$ that acts as a 3-fold rotation, best visualized using a pants decomposition of $F$ provided by a primitive triple of disks, $D'$, $E'$, and $F'$ in $W$ (see Figure (10) from \cite{ChoG2GGLS}) . With the curves $a$, $x$, $b$, and $y$, paying keen attention to orientations and where intersections between curves occur relative to each other, the result of cutting along the disks $D'$, $E'$, and $F'$ is shown in Figure \ref{figure:deltabasepants}. Then $\delta$ rotates each of the pairs of pants clockwise by $2\pi/3$, cyclically permuting the disks $D'$, $F'$, and $E'$ (shown in the bottom left, top, and bottom right, resp.), as well as the disks $D$, $E$, and $F$ (shown in blue, brown, and red, resp.) in a primitive triple determined by $D'$, $E'$, and $F'$. We can then analyze the intersections of the images $\delta(a)$, $\delta(x)$, $\delta(b)$, and $\delta(y)$, respectively shown in Figures \ref{figure:deltaofa}, \ref{figure:deltaofx}, \ref{figure:deltaofb}, and \ref{figure:deltaofy}, with all four of the original curves, $a$, $x$, $b$, and $y$. The columns of $\delta_*$ are then, again, produced from Equation \ref{eqn:vector}.

Finally, when $p \geq 4$, $\sigma$ acts similarly to $\gamma$, but from the perspective of $V$ instead of $W$. We verify that Figure \ref{figure:sigma} is an accurate representation of the basis curves. The curves $x$ and $y$ are parallel to $\bd E'$ and $\bd E$, respectively, as $x$ bounds a disk in the stabilized portion of $W$ and $y$ bounds a disk in the stabilized portion of $V$. The curves $a$ and $b$ are each disjoint from $E \cup E'$, $a$ bounds a disk in $W$ that intersects $\bd D$ $p$ times, and $b$ intersects each of $\bd D$, $a$, and the boundary of the band sum of $D$ and $E$ along $\bd E'$ exactly once. However, while the relative orientations of $a$ and $b$ and of $x$ and $y$ are clear, the orientations of these pairs relative to each other is potentially ambiguous. To further settle this discrepancy, observe that if any orientation were to be placed on $\bd D$, then $sign ( \io(a, \bd D) ) = sign( \io(x, \bd D) )$ at every point of intersection (see Figure \ref{figure:disksandgens}), so the orientations in Figure \ref{figure:sigma} are consistent. Finally, then, the columns of $\sigma_*$ are produced from Equation \ref{eqn:vector}.
\end{proof}

\section{The star map}
\label{section:Star}

We are now in a position to investigate structures of the subgroups of $GL(4, \mathbb{Z})$ formed by the images of the star map and the kernel of this map.

Let \[\mathcal{A}_p = \set{ A = \begin{pmatrix} kp+1 & \ell \\ mp & np + \Delta \end{pmatrix} \, :  \,  k, \ell, m, n \in \mathbb{Z}, \, \det A = \Delta \in \set{\pm 1}}.\] For $A = \begin{pmatrix} kp+1 & \ell \\ mp & np + \Delta \end{pmatrix}  \in \mathcal{A}_p$, let $D_A = \begin{pmatrix} k - \Delta n & \Delta m \\ m & 0 \end{pmatrix}$, and let \[\mathcal{S}_p = \set{ \pm \begin{pmatrix} A & \varepsilon D_A \\ 0 & (A^{T})^{-1} \end{pmatrix} \, : \, A \in \mathcal{A}_p},\] where $\varepsilon = -1$ if $p = 3$, and $\varepsilon = 1$ otherwise. Further, let \[\mathcal{B}_p = \set{ A = \begin{pmatrix} kp+1 & \ell \\ mp & np + \Delta \end{pmatrix} \, :  \, A \in \mathcal{A}_p, \, |mp | \leq 2|kp + 1| \mbox{ or } p| \ell |  \leq  2 | np + \Delta |}.\]

\begin{thm} \label{thm:subgroups} For all $p \geq 2$, $*(\mathcal{G}_p)$ is a subgroup of $\mathcal{S}_p$. 
\end{thm}

\begin{proof}
First, suppose $M_i = \pm \begin{pmatrix} C_i & \varepsilon D_{C_i} \\ 0 & (C_i^T)^{-1} \end{pmatrix} \in \mathcal{S}_p$, $(i = 1, 2)$, where $\varepsilon = -1$ when $p=3$, and $\varepsilon=1$ otherwise. Then, for $p \neq 3$, $M_1 M_2 = \pm \begin{pmatrix} C_1C_2 & C_1 D_{C_2} + D_{C_1} \left(C_2^T\right)^{-1} \\ 0 & \left( \left( C_1 C_2 \right)^T \right)^{-1} \end{pmatrix}$. If $C_i = \begin{pmatrix} k_i p + 1 & \ell_i \\ m_i p & n_i p + \Delta_i \end{pmatrix}$, $(i = 1, 2)$, then 
\begin{eqnarray*} C_1 C_2 & = & \begin{pmatrix} (k_1 p + 1)(k_2 p + 1) + \ell_1 m_2 p & \ell_2 (k_1p + 1) + \ell_1 (n_2 p + \Delta_2) \\ m_1 p (k_2 p + 1) + m_2 p (n_1 p + \Delta_1) & m_1 p \ell_2 + (n_1p + \Delta_1)(n_2 p + \Delta_2) \end{pmatrix} \\
& = & \begin{pmatrix} k_3p + 1 &  \ell_3 \\ m_3 p & n_3 p + \Delta_3 \end{pmatrix}, \end{eqnarray*}
with $k_3 = (k_1 k_2 p + k_1 + k_2 + \ell_1 m_2)$, $\ell_3 = \ell_2 (k_1p + 1) + \ell_1 (n_2 p + \Delta_2)$, $m_3 = (m_1(k_2 p + 1) + m_2 ( n_1 p + \Delta_1))$, $n_3 = (m_1 \ell_2 + n_1 n_2 p + n_1 \Delta_2 + n_2 \Delta_1)$, and $\Delta_3 = \Delta_1 \Delta_2 = \det(C_1 C_2)$. It is routine, albeit tedious, to verify the equation $C_1 D_{C_2} + D_{C_1} \left(C_2^T\right)^{-1} = D_{C_1 C_2}$. When $p=3$, we have, instead, $M_i = \pm \begin{pmatrix} C_i & -D_{C_i} \\ 0 & (C_i^T)^{-1} \end{pmatrix} \in \mathcal{S}_3$, and we have $M_1 M_2 = \pm \begin{pmatrix} C_1C_2 & - C_1 D_{C_2} - D_{C_1} \left(C_2^T\right)^{-1} \\ 0 & \left( \left( C_1 C_2 \right)^T \right)^{-1} \end{pmatrix}$. Note also we have $- C_1 D_{C_2} - D_{C_1} \left(C_2^T\right)^{-1} = - \left( C_1 D_{C_2} + D_{C_1} \left(C_2^T\right)^{-1} \right) = -D_{C_1 C_2}$. Thus, $\mathcal{S}_p$ forms a subgroup of $GL(4, \mathbb{Z})$. 

Finally, observe that the images of each generator, $\alpha_*, \beta_*$, $\gamma_*$, $\delta_*$, $\rho_*$, and $\sigma_*$ all have the form of $\mathcal{S}_p$.
\end{proof}

In order to more easily study $*(\mathcal{G}_p)$, we consider a projection of $\mathcal{S}_p$ into $ \mathcal{A}_p$. Let $q : \mathcal{S}_p \to \mathcal{A}_p$ be the projection map onto the top left $2 \times 2$ block. 

\begin{lem} \label{lem:qinjective} The map $q$ is an injective homomorphism.
\end{lem}

\begin{proof} It is evident that $q$ is a homomorphism, so it suffices to prove injectivity. Suppose $q(M_1) = \omega_1A_1 = \omega_2A_2 = q(M_2)$, where $A_i = \begin{pmatrix} k_i p + 1 \ & \ell_i \\ m_i p & n_i p + \Delta_i \end{pmatrix}$, and $\omega_i \in \set{\pm 1}$. 

First, observe that $\Delta_1 = \det(A_1) = \det(\omega_1 A_1) = \det(\omega_2 A_2) = \det(A_2) = \Delta_2$, regardless of the values of the $\omega_i$.

Now, suppose $p > 2$. Then, $\omega_1 = \omega_2$, or else $(\omega_1 k_1 - \omega_2 k_2)p = \omega_2 - \omega_1$ would force $p$ to divide $\pm 2$. Thus, we get immediately that $\ell_1 = \ell_2$, $m_1 = m_2$, and $k_1 = k_2$. Similarly, $(n_1 - n_2)p = \Delta_2 - \Delta_1 = 0$, in which case $n_1 = n_2$. It follows that $M_1 = M_2$.

On the other hand, suppose $p = 2$. If $\omega_1 = \omega_2$, then $M_1 = M_2$, as above. If $\omega_1 \neq \omega_2$, then we find that $m_2 = -m_1$, and $\omega_1(2k_1 + 1) = \omega_2(2k_2 + 1)$, which implies that $k_1 + k_2 = -1$. Also, $\omega_1(2 n_1 + \Delta_1) = \omega_2(2n_2 + \Delta_2)$, from which we find $n_1 + n_2 = \pm 1$, depending on whether $\Delta_1 = -1$ or $\Delta_1 = 1$, respectively. In the first case, $\omega_1(k_1 - \Delta_1 n_1) = (-\omega_2)((-k_2 - 1) - (-1) (1-n_2)) = (\omega_2)((k_2 + 1) + (-1) (1-n_2)) = \omega_2(k_2 + n_2) = \omega_2(k_2 - \Delta_2n_2)$. In the latter case, $\omega_1(k_1 - \Delta_1 n_1) = (-\omega_2)((-k_2 - 1) - (1)(-1 - n_2)) = (\omega_2)(k_2 - n_2) = \omega_2(k_2 - \Delta_2n_2)$. Either way, this implies that $M_1 = M_2$.
\end{proof}

We will show that $\beta^2$ generates the kernel of the star map. Let $\mathcal{G}_p' = \mathcal{G}_p / \langle \langle \beta^2 \rangle \rangle$. 

\begin{thm} \label{thm:kernel} The kernel of the star map is the normal subgroup generated by $\beta^2$.
\end{thm}

\begin{proof} Let $r : \mathcal{G}_p \to \mathcal{G}_p'$ be the natural projection, and define $*' : \mathcal{G}_p' \to GL(4, \mathbb{Z})$ in the natural way so that $* = *' \circ r$, so we have

\[ \mathcal{G}_p \stackrel{r}{\longrightarrow} \mathcal{G}_p' \stackrel{*'}{\longrightarrow} GL(4, \mathbb{Z}) \stackrel{q}{\longrightarrow} GL(2, \mathbb{Z}).\]

We will again abuse notation slightly and refer to $r(\beta)$, $r(\gamma)$, and $r(\sigma)$ again as simply $\beta$, $\gamma$, and $\sigma$, respectively, and to $*'(\beta)$, $*'(\gamma)$, and $*'(\sigma)$ again as $\beta_*$, $\gamma_*$, and $\sigma_*$, respectively.

As the projection $q$ is injective, it will suffice to prove that $*'$ is injective in each of the separate cases, $p=2$, $p=3$, and $p \geq 4$.

\begin{description}
\item[Case 1] Suppose $p = 2$, and $\mathcal{G}_2 = \langle \, \beta, \rho, \gamma \ \mid \ \rho^4 = \gamma^2 = (\gamma \rho)^2 = \rho^2 \beta \rho^2 \beta^{-1} = 1 \, \rangle$. 
\begin{claim}
\label{claim:G2NormalForm} In $\mathcal{G}_2'$, any word can be expressed in the form 
\[ (\rho^2)^{\omega } \gamma^{\varepsilon_{n+1}} (\eta_{k_n}\gamma^{\varepsilon_n}\rho)\cdots (\eta_{k_2}\gamma^{\varepsilon_2}\rho) (\eta_{k_1}\gamma^{\varepsilon_1}\rho) \eta_{k_0} \gamma^{\varepsilon_{0}}, \]

where $\eta_k = (\beta \gamma)^{k-1} \beta$, (with $k$ factors of $\beta$, and $\eta_0$ the empty word), $n \geq 0$,  $k_1, k_2, \dots, k_{n-1} \geq 1$, $k_0, k_n \geq 0$, and $\omega , \varepsilon_0, \dots, \varepsilon_{n+1} \in \set{0, 1}$.

\end{claim}

\begin{proof}
In $\mathcal{G}'_2$, both $\beta$ and $\gamma$ have order 2, so $\beta$ and $\gamma$ may appear only with exponent 0 or 1. Since $\rho$ has order 4, and $\rho^2$ commutes with all other elements, wherever $\rho$ appears, it may appear with exponent 1, with the possible exception of a $\rho^2$ at the left of the word. (If $\rho^3$ appears, we rewrite it as $\rho^2 \rho$ and move the $\rho^2$ to the left.)

Now, $\rho$ and $\gamma$ commute, up to a factor of $\rho^2$, so we may push every instance of $\gamma$ to the left of any adjacent $\rho$. Any two $\rho$'s will be separated by a word that alternates between $\beta$ and $\gamma$. However, there cannot be a $\gamma$ on the left of this alternating word, as it would have been pushed past the $\rho$ on the left, so in particular there must be at least one $\beta$ between any two $\rho$'s. Thus, between any two instances of $\rho$ will appear $\eta_k \gamma^\varepsilon$ with $k \geq 1$, and $\varepsilon \in \set{0, 1}$.

The word, then, may either have $\rho$ on the right, or not. In the former case, $\varepsilon_0 = k_0 = 0$. Otherwise, the sub-word to the right of the rightmost $\rho$ is an alternating word in $\gamma$ and $\beta$, with no $\gamma$ on the left, and with either a $\beta$ or a $\gamma$ on the right.

Finally, between the leftmost single factor of $\rho$ and a possible factor of $\rho^2$, there is a (possibly empty) alternating word in $\gamma$ and $\beta$ that may have either on the left and either on the right.
\end{proof}

\begin{claim}
\label{claim:allnonnegallnonpos}
If $k_1, \dots, k_{n} \geq 1$, the matrix $\left( q \circ *^{'} \right) \left( \left( \eta_{k_{n}}\gamma^{\varepsilon_{n}}\rho \right) \cdots \left( \eta_{k_1}\gamma^{\varepsilon_1}\rho \right) \right) = \begin{pmatrix} a & b \\ c & d \end{pmatrix}$ has one of the following properties:
\begin{enumerate}
\item $a$, $c$, $d$ are positive and $b$ is nonnegative, or
\item $a$, $c$, $d$ are negative and $b$ is nonpositive.
\end{enumerate} 
\end{claim}

\begin{proof}
The proof is by induction on $n$. We have
\[ \left( q \circ *^{'} \right) \left( \eta_{k_{1}} \rho \right) = \begin{pmatrix} 2k_1 - 1 & k_1 \\ 2 & 1 \end{pmatrix} ,\]
and
\[ \left( q \circ *^{'} \right) \left( \eta_{k_{1}}\gamma \rho \right) = \begin{pmatrix} 1-2k_1 & 1-k_1 \\ -2 & -1 \end{pmatrix}.\]
If $\left( q \circ *^{'} \right) \left( \left( \eta_{k_{n-1}}\gamma^{\varepsilon_{n-1}}\rho \right) \cdots \left( \eta_{k_1}\gamma^{\varepsilon_1}\rho \right) \right) = \begin{pmatrix} a & b \\ c & d \end{pmatrix}$ has one of the properties, then $\left( q \circ *^{'} \right) \left( \left( \eta_{k_{n}}\rho \right) \left( \eta_{k_{n-1}}\gamma^{\varepsilon_{n-1}}\rho \right) \cdots \left( \eta_{k_1}\gamma^{\varepsilon_1}\rho \right) \right)$ is
\begin{align*} \begin{pmatrix} 2k_n -1 & k_n \\ 2 & 1 \end{pmatrix} \begin{pmatrix} a & b \\ c & d \end{pmatrix} & =  \begin{pmatrix} (2k_n - 1)a + k_n c & (2k_n - 1)b + k_n d \\ 2a + c & 2b + d \end{pmatrix},\end{align*}  
and $\left( q \circ *^{'} \right) \left( \left( \eta_{k_{n}} \gamma \rho \right) \left( \eta_{k_{n-1}}\gamma^{\varepsilon_{n-1}}\rho \right) \cdots \left( \eta_{k_1}\gamma^{\varepsilon_1}\rho \right) \right)$ is
\begin{align*} \begin{pmatrix} 1-2k_n  & 1-k_n \\ -2 & -1 \end{pmatrix} \begin{pmatrix} a & b \\ c & d \end{pmatrix} & = \begin{pmatrix} (1-2k_n)a + (1-k_n) c & (1-2k_n)b + (1-k_n) d \\ -2a - c & -2b - d \end{pmatrix},   
\end{align*}
both of which also have one of these properties.
\end{proof}

In light of Claim \ref{claim:allnonnegallnonpos}, we will call a word of the form $\left( \eta_{k_{n}}\gamma^{\varepsilon_{n}}\rho \right) \cdots \left( \eta_{k_1}\gamma^{\varepsilon_1}\rho \right)$ with $k_1, \dots, k_n \geq 1$ a \emph{non-trivial block}.

We can now prove that $ \left( q \circ *^{'} \right) $ is injective. Let $w$ be a generic word in $\mathcal{G}_2'$, as in Claim \ref{claim:G2NormalForm}. As $\left(q \circ *' \right)(\rho^2) = \begin{pmatrix} -1 & 0 \\ 0 & -1 \end{pmatrix}$, it will suffice to assume that $\omega  = 0$, and show that $\left( q \circ *' \right)(w)$ is neither plus nor minus the identity unless $w$ is trivial. Further, triviality is preserved by conjugation, so we may cyclically permute as desired. First, we will replace $w$ with a word of the form
\[ (\eta_{k_n}\gamma^{\varepsilon_n}\rho) (\eta_{k_{n-1}}\gamma^{\varepsilon_{n-1}}\rho)\cdots (\eta_{k_2}\gamma^{\varepsilon_2}\rho) (\eta_{k_1}\gamma^{\varepsilon_1}\rho) \eta_{k_0} \gamma^{\varepsilon_{0}}, \]
where $n \geq 0$,  $k_1, k_2, \dots, k_{n-1} \geq 1$, $k_0, k_n \geq 0$, and $\varepsilon_0, \dots, \varepsilon_{n} \in \set{0, 1}$

If $n = 0$, then up to cyclic permutation $w = \eta_{k_0} \gamma^{\varepsilon}$, for $\varepsilon \in \set{0, 1}$. This is either $\begin{pmatrix} 1 & k_0 \\ 0 & 1 \end{pmatrix}$ if $\varepsilon = 0$, or $\begin{pmatrix} 1 & 1-k_0 \\ 0 & -1 \end{pmatrix}$ if $\varepsilon=1$, giving (plus or minus) the identity only if $k_0 = \varepsilon = 0$ so that $w$ is the identity. If $n = 1$, and $k_0 = k_1 = 0$, then up to cyclic permutation, $w = \rho \gamma^\varepsilon$, for $\varepsilon \in \set{0, 1}$, both of which are demonstrably nontrivial under $\left( q \circ *'\right)$.

Otherwise, if $n > 1$ or if $k-1$ and at least one of $k_0$ or $k_1$ is positive, then $w$ can be cyclically permuted to be a product of a nontrivial block with a word of the form $\eta_k \gamma^{\varepsilon} \rho^\tau$ for some $k \geq 0$, $\varepsilon, \tau \in \set{0, 1}$. 

This leaves only a few cases left to check:
\begin{enumerate}
\item If $\tau = 0$, then it is easy to check that whether $k=0$ and $\varepsilon = 1$, or $k>0$ (and $\varepsilon = 0$, or $\varepsilon =1$), multiplying (on the right) a matrix $\begin{pmatrix} a & b \\ c & d \end{pmatrix}$ as in Claim \ref{claim:allnonnegallnonpos} will have a bottom left entry of $c$, which is nonzero.

\item On the other hand, suppose $\tau = 1$.

If $k > 0$, then in fact $w$ is itself cyclically equivalent to a nontrivial block. And otherwise, we compute
\[ \begin{pmatrix} a & b \\ c & d \end{pmatrix} \cdot \left(q \circ *'\right) \left( \rho \right) = \begin{pmatrix} a & b \\ c & d \end{pmatrix} \begin{pmatrix} 1 & 1 \\ -2 & -1 \end{pmatrix} = \begin{pmatrix} a-2b & a-b \\ c-2d & c-d \end{pmatrix},\]
which could only be (plus or minus) the identity if $a=b$ and $c=2d$,  
and
\[ \begin{pmatrix} a & b \\ c & d \end{pmatrix} \cdot \left(q \circ *'\right) \left( \gamma \rho \right) = \begin{pmatrix} a & b \\ c & d \end{pmatrix} \begin{pmatrix} -1 & 0 \\ 2 & 1 \end{pmatrix} = \begin{pmatrix} 2b-a & b \\ 2d-c & d \end{pmatrix}\]
which could only be (plus or minus) the identity if $b=0$. In both cases, it would be necessary that $-a = d = \pm1$, which is impossible as $a$ and $d$ have the same sign.
\end{enumerate}

\item[Case 2] Suppose $p = 3$, and $\mathcal{G}_3 = \langle \, \alpha \ \mid \ \alpha^2 = 1 \, \rangle \bigoplus \langle \, \beta, \, \delta, \, \gamma \ \mid \ \delta^3 = \gamma^2 = (\gamma \delta)^2 = 1 \, \rangle$.

\begin{claim}
\label{claim:G3NormalForm} In $\mathcal{G}_3'$, any word can be expressed in the form 
\[\alpha^{\varepsilon_3} \beta^{\varepsilon_2} \eta_n \cdots \eta_1 \beta^{\varepsilon_1},\] where $\varepsilon_i \in \set{0, 1}$ for $i = 1, 2, 3$ and $\eta_i = \beta \gamma^{e_i} \delta^{f_i}$ with $e_i \in \{0,1\}$ and $f_i \in \{0,1,2\}$ for $i = 1, 2, \ldots, n$, but $(e_i, f_i) \neq(0,0)$.
\end{claim}

\begin{proof}
Since $\alpha$ commutes with the other generators, and has order 2, we may assume that $\alpha$ appears at most once at the left of a word. Otherwise, we can think of the general word for an element of $\mathcal{G}'_3$ as being alternating powers of $\gamma$ and $\delta$ with individual $\beta$'s spread between them. Since $\gamma$ has order 2 and $\delta$ has order 3, this means we have sections of the word of the form $\eta = \beta \gamma^{b_1}\delta^{c_1} \cdots \gamma^{b_k} \delta^{c_k}$. However, the relations for $\mathcal G_3$ tell us that $\gamma \delta  = \delta^{2} \gamma$, so we can move all of the instances of $\gamma$ to the immediate right of $\beta$ in $\eta$, at the expense of changing the power of $\delta$. Then $\eta = \beta \gamma^{e} \delta^{f}$ for $e \in \{0,1\}$ and $f \in \{0,1,2\}$. On the right, there may either be a $\beta$ or not, so $\varepsilon_1$ is either 1 or 0, respectively, and on the left (immediately to the right of $\alpha$, if present), there may either be a $\beta$ or not, so $\varepsilon_2$ is either 1 (to cancel with the $\beta$ in $\eta_n$), or 0. Thus, we have the normal form, as claimed. 
\end{proof}

\begin{claim}
\label{claim:allposallneg}
No matrix $\left( q \circ *^{'} \right) \left( \eta_n \cdots \eta_1 \right)  $ is equal to $I$, $q(\alpha_*) = \begin{pmatrix} -1 & 0 \\ 0 & -1 \end{pmatrix}$, $q(\beta_*) = \begin{pmatrix} 1 & 0 \\ 0 & -1 \end{pmatrix}$, or $q(\alpha_* \beta_*) = \begin{pmatrix} -1 & 0 \\ 0 & 1 \end{pmatrix}$. \end{claim}

\begin{proof}
Each $\eta_i$ will be of the form $\beta \gamma$, $\beta \delta$, $\beta \delta^2$, $\beta \gamma \delta$, or $\beta \gamma \delta^2$. Observe that \[q (\beta_* \gamma_*) = \begin{pmatrix} 1 & 1  \\ 0 & 1\end{pmatrix}, \qquad q (\beta_* \delta_*) = \begin{pmatrix} -2 & -1  \\ -3 & -1\end{pmatrix}, \qquad q  (\beta_* \delta_*^2) = \begin{pmatrix} 1 & 1  \\ 3 & 2\end{pmatrix},\]

\[q (\beta_* \gamma_* \delta_*) = \begin{pmatrix} 1 & 0  \\ 3 & 1\end{pmatrix},\qquad  \mbox{ and } \qquad q (\beta_* \gamma_* \delta_*^2) = \begin{pmatrix} -2 & -1  \\ -3 & -2\end{pmatrix}.\]

Any power of $q( \beta_* \gamma_*)$ or $q(\beta_* \gamma_* \delta_*)$ will evidently satisfy the statement. 

If each of the $\eta_i \in \set{\beta \delta, \, \beta \delta^2, \, \beta \gamma \delta^2}$, then $(q \circ *)(\eta_i)$ will either be a positive matrix (all entries are strictly positive) or a negative matrix (all entries are strictly negative). Any product of positive matrices is a positive matrix, any product of two negative matrices will be a positive matrix, and any product between one positive matrix and one negative matrix will be a negative matrix. Any product of a positive or a negative matrix with any power of $q(\beta_* \gamma_*)$ or $q(\beta_* \gamma_* \delta_*)$ will again be positive or negative. And finally, $q(\beta_* \gamma_*) \cdot q(\beta_* \gamma_* \delta_*) = \begin{pmatrix} 4 & 1 \\ 3 & 1 \end{pmatrix}$ and $q(\beta_* \gamma_* \delta_*) \cdot q( \beta_* \gamma_*) = \begin{pmatrix} 1 & 1 \\ 3 & 3 \end{pmatrix}$ are both positive matrices. Observe that none of $I$, $q(\alpha_*)$, $q(\beta_*)$, or $q(\alpha_* \beta_*)$ are positive or negative matrices. The statement follows.
\end{proof}

Now, any word of the form from Claim \ref{claim:G3NormalForm} is cyclically conjugate to $w = \alpha^{\varepsilon_3} \beta^{\varepsilon_2} \eta_n \cdots \eta_1$, so Claim \ref{claim:allposallneg} shows that $q(w_*)$ will never be trivial unless $w_*$ is trivial.

\item[Case 3] Suppose $p \geq 4$, and $\mathcal{G}_p = \langle \, \alpha \ \mid \ \alpha^2 = 1 \, \rangle \bigoplus \langle \, \beta, \, \gamma, \, \sigma \ \mid \ \gamma^2 = \sigma^2 = 1\, \rangle$. 

Suppose that $q(w_*) = I$, where $w = g_1 g_2 \cdots g_m$ is a reduced word in $\mathcal{G}_p'$. In fact, we may assume that $g_i \in \set{\beta, \gamma, \sigma}$ for all $i$, and that $m$ is even, since $\det(q(\beta_*)) = \det(q(\gamma_*)) = \det(q(\sigma_*)) = -1$, and $\alpha$ commutes with the other generators. Then, as $q(\beta_*)^2 = I$, we may insert a pair of $\beta$'s after every other letter, so that \[q({g_1}_* \beta_* \beta_* {g_2}_* {g_3}_* \beta_* \beta_* {g_4}_* \cdots {g_{m-2}}_* {g_{m-1}}_* \beta_* \beta_* {g_{m}}_* ) = I.\] That is to say, $q({v_1}_* {v_2}_* \cdots {v_n}_*) = I$, where $v_i \in \set{ \beta \gamma, \gamma \beta, \beta \sigma, \sigma \beta}$, after cancelling any remaining $\beta_*^2$'s. As $\left(q(\beta_* \gamma_*)\right)^{-1} = q(\gamma_* \beta_*)$, and $\left(q(\beta_* \sigma_*)\right)^{-1} = q(\sigma_* \beta_*)$, this means that $q({v_1}_* {v_2}_* \cdots {v_n}_*)$ is a word generated by $q(\beta_* \gamma_*) = \begin{pmatrix} 1 & 1 \\ 0 & 1 \end{pmatrix}$ and $q(\beta_* \sigma_*) = \begin{pmatrix} 1 & 0 \\ p & 1 \end{pmatrix}$. By the second corollary of \cite{LynUllPR2x2MGFP}, $\begin{pmatrix} 1 & 1 \\ 0 & 1 \end{pmatrix}$ and $\begin{pmatrix} 1 & 0 \\ p & 1 \end{pmatrix}$ generate a free group, so that the only way $q({v_1}_* {v_2}_* \cdots {v_n}_*) = I$ is if $v_1 v_2 \cdots v_n$ is trivial, which in turn implies that $w$ was trivial. We conclude that $*'$ is injective, as claimed. 

\end{description}

This completes the proof of Theorem \ref{thm:kernel}.
\end{proof}

We can characterize the subgroup $*(\mathcal{G}_p)$ for each $p$.

\begin{thm} \label{thm:preciseimages} For $p \in \set{2, 3, 4}$, \[*(\mathcal{G}_p) = \mathcal{S}_p,\] and if $p \geq 5$, \[*(\mathcal{G}_p) = \set{ \pm \begin{pmatrix} B & D_B \\ 0 & (B^T)^{-1} \end{pmatrix} \, : \, B \in \mathcal{B}_p }.\]
\end{thm}

\begin{rem} \label{rem:inequalitytrivial} In fact, even for $p \in \set{2, 3, 4}$, $*(\mathcal{G}_p) = \set{ \pm \begin{pmatrix} B & D_B \\ 0 & (B^T)^{-1} \end{pmatrix} \, : \, B \in \mathcal{B}_p }$, as the condition that $|mp | \leq 2|kp + 1|$ or $p| \ell |  \leq  2 | np + \Delta |$ is trivially satisfied by all matrices in $\mathcal{S}_p$. To see this, note that if neither inequality were satisfied, then 
\[|pm \cdot \ell| = \frac{1}{p} \cdot |pm| \cdot p| \ell | \geq \frac{1}{p} ( 2|kp+1| + 1)( 2|np+\Delta| + 1 ) \geq \frac{4}{p} (|kp+1|)( |np+\Delta| ) + \frac{5}{p}.\] When $2 \leq p \leq 4$, this inequality would preclude $\Delta = \pm 1$. However, for $p \geq 5$, the inequalities constitute a nontrivial condition, carving out a proper subgroup of $\mathcal{S}_p$. 
\end{rem}

In order to prove Theorem \ref{thm:preciseimages}, we develop some structure. In light of Lemma \ref{lem:qinjective}, it will suffice to consider the $2 \times 2$ matrices, $q(\mathcal{S}_p)$, which forms a subgroup of $\mathcal{A}_p$. 

Let $\mathcal{H}_p = q\left( *(\mathcal{G}_p) \right)$. We first prove the following lemma.

\begin{lem} \label{lem:inequalities} For all $p \geq 4$,  \[\mathcal{H}_p \subseteq  \mathcal{B}_p .\]
\end{lem} 

\begin{proof} First, observe that the map from $\mathcal{G}_p$ to $\mathcal{H}_p$ factors through \[\mathcal{G}_p'  = \mathcal{G}_p / \langle \langle \beta^2 \rangle \rangle. \]

In this case, $p \geq 4$, so $\mathcal{G}_p' = \langle \, \alpha \mid \alpha^2 = 1 \, \rangle \oplus \langle \, \beta, \, \gamma, \, \sigma \mid \beta^2 = \gamma^2 = \sigma^2= 1 \, \rangle$. We will abuse notation slightly and refer the generators in $\mathcal{G}_p'$ again as simply $\alpha$, $\beta$, $\gamma$, and $\sigma$.

To prove this lemma, we will induct on the reduced word length of elements $w'$ of $\mathcal{G}_p'$. Without loss of generality, we may ignore $\alpha$, and assume that all matrices are products of images of $\beta$, $\gamma$, and $\sigma$. Consider a matrix $A = \begin{pmatrix} a & b \\ c & d \end{pmatrix} \in \mathcal{H}_p$ as a M\"{o}bius transformation, acting on the extended real line, $\overline{\mathbb{R}} = \mathbb{R} \cup \set{\infty}$. Observe that $q(\gamma_*)(x) = -x - 1$, $q(\sigma_*)(x) = \frac{x}{-px - 1}$, and $q(\beta_*)(x) = -x$. For a reduced nontrivial word $w' \in \mathcal{G}_p'$, let $LL(w')$ refer to the smallest, left-most subword that contains either a $\gamma$ or a $\sigma$. Thus, $LL(w') \in \set{\gamma, \sigma, \beta \gamma, \beta \sigma, \emptyset }$, and $LL(\beta) = \emptyset$. 

Observe that $q(\gamma_*)$ has fixed points $-\frac{1}{2}$ and $\infty$, and exchanges the two complementary intervals of $\overline{\mathbb{R}} \rmv \set{-\frac{1}{2}, \infty}$. Similarly, $q(\sigma_*)$ has fixed points $-\frac{2}{p}$ and $0$,  $q(\beta_*)$ has fixed points $0$ and $\infty$, and both exchange the two complementary intervals of their respective fixed points. So, let 
\[ I_\gamma^- = \left( -\infty, -\frac{1}{2} \right] \cup \set{\infty}, \qquad I_\gamma^+ = \left[\frac{1}{2}, \infty \right) \cup \set{\infty}, \qquad I_\gamma  = I_\gamma^- \cup I_\gamma^+,  \]
and
\[ I_\sigma^- = \left[ -\frac{2}{p}, 0 \right], \qquad I_\sigma^+ = \left[ 0, \frac{2}{p} \right], \qquad I_\sigma = I_\sigma^- \cup I_\sigma^+.\]
Finally, set $P = \set{0, \infty}$.

We make an initial observation. If $p > 4$, then $I_\gamma$ and $I_\sigma$ are disjoint. If $p = 4$, then $\frac{2}{p} = \frac{1}{2}$, and the two intervals $I_\sigma$ and $I_\gamma$ overlap at endpoints. However,

\begin{claim}  \label{claim:overlap} If $p = 4$,  $q(w_*)(0) \neq \pm \frac{1}{2}$  and $q(w_*)(\infty) \neq \pm \frac{1}{2}$. 
\end{claim}

\begin{proof}
If $q(w_*) = A = \begin{pmatrix} a & b \\ c & d \end{pmatrix}$, then $A(0) = \frac{b}{d}$ and $A(\infty) = \frac{a}{c}$. But if $\frac{b}{d} = \pm \frac{1}{2}$, this would mean that $d = \pm 2b$, and if $\frac{a}{c} = \pm \frac{1}{2}$, then $c = \pm 2a$. However, since $\pm A \in \mathcal{A}_p$, $\pm a = kp + 1$ and $\pm b = \ell$ for some $k, \ell \in \mathbb{Z}$, $\pm c = mp$, and $\pm d = np + \Delta$ for some $m, n \in \mathbb{Z}$. Neither of these are then possible if $p = 4$.
\end{proof}

\begin{claim} \label{claim:images} For a reduced nontrivial word $w \in \mathcal{G}_p$, 
\begin{enumerate}
\item \label{case:G} if $LL(w') = \gamma$, then $q(w_*)(P) \subseteq I_\gamma^-$, and 
\item \label{case:S} if $LL(w') = \sigma$, then $q(w_*)(P) \subseteq I_\sigma^-$.
\end{enumerate}

\end{claim}

\begin{proof} We prove the claim by induction on reduced wordlength of words in $\mathcal{G}_p'$. 

Suppose $w \in \mathcal{G}_p$ and $|w'|= 1$. If $LL(w') = \gamma$, then, $q(w_*) = q(\gamma_*)$. Thus, $q(\gamma_*)$ fixes $\infty$, and $q(\gamma_*)(0) = -1 \in I_\gamma^-$, which establishes (\ref{case:G}). If $LL(w') = \sigma$, then $q(w_*) = q(\sigma_*)$. Thus, $q(\sigma_*)$ fixes $0$, and $q(\sigma_*)(\infty) = -\frac{1}{p} \in I_\sigma^-$, which establishes (\ref{case:S}).

If $w \in \mathcal{G}_p$ and $|w'| = 2$, then there are 6 possible words that $w'$ can represent. Four of them satisfy one of the hypotheses of Claim \ref{claim:images}: $w' = \gamma \sigma$ and $\gamma \beta$ satisfy the hypothesis of (\ref{case:G}), while  $w' = \sigma \gamma$ and $w' =  \sigma \beta $ satisfy the hypothesis of (\ref{case:S}). Each of these may be checked directly, as $q(\gamma_* \sigma_*) = \begin{pmatrix} -p+1 & -1 \\ p & 1 \end{pmatrix}$, $q(\gamma_* \beta_*)= \begin{pmatrix} 1 & -1 \\ 0 & 1 \end{pmatrix}$, $q(\sigma_* \gamma_*) = \begin{pmatrix} 1 & 1 \\ -p & -p+1 \end{pmatrix}$, and $q(\sigma_* \beta_*)= \begin{pmatrix} 1 & 0 \\ -p & 1 \end{pmatrix}$.

Now, assume that for $n > 1$, the claim holds for all words with length less than $n$, and suppose $w' \in \mathcal{G}_p'$ with $|w'| = n$. 
In order for the hypotheses of Claim \ref{claim:images} to apply, it must be that $w' = gv$, for some $v \in \mathcal{G}_p'$ with $|v| = n-1$, and $g \in \set{\gamma, \sigma}$. 

If $g = \gamma$, then $LL(v) \in \set{\sigma,  \beta \sigma,  \beta \gamma, \emptyset}$. Suppose $LL(v) = \emptyset$. Then, $v = \beta$, so $q(v_*)$ fixes both $0$ and $\infty$, so $q(w_*)(P) = q(\gamma_*)(P) \subseteq I_\gamma^-$. Suppose $LL(v) \in \set{\sigma, \beta \sigma}$. Then, inductively, $q(v_*)(P) \in I_\sigma^{\pm}$, and by Claim \ref{claim:overlap}, $q(v_*)(P)$ is in the complement of $[-\infty, -\frac{1}{2}]$, so that $q(w_*)(P) \subseteq I_\gamma^-$. Suppose finally that $LL(v) = \beta \gamma$. Then inductively, $q(v_*)(P) \subseteq I_\gamma^+$, so that $q(w_*)(P) \subseteq I_\gamma^-$. 

Similarly, if $g = \sigma$, then
$LL(v) \in \set{\gamma, \beta \gamma, \beta \sigma, \emptyset}$. Suppose $LL(v) = \emptyset$. Then, $v = \beta$, so $q(v_*)$ fixes both $0$ and $\infty$, so $q(w_*)(P) = q(\sigma_*)(P) \subseteq I_\sigma^-$. Suppose $LL(v) \in \set{\gamma, \beta \gamma}$. Then, inductively, $q(v_*)(P) \in I_\gamma^{\pm}$, and by Claim \ref{claim:overlap}, $q(v_*)(P)$ is in the complement of $[-\frac{2}{p}, 0]$, so that $q(w_*)(P) \subseteq I_\sigma^-$. Suppose, finally, that $LL(v) = \beta \sigma$. Then, inductively, $q(v_*)(P) \subset I_\sigma^+$, so that $q(w_*)(P) \subseteq I_\sigma^-$.
\end{proof}

We are now in a position to proceed with the proof of Lemma \ref{lem:inequalities}. Let $w \in \mathcal{G}_p$ with reduced word $w' \in \mathcal{G}_p'$, so that $q(w_*) = \begin{pmatrix} kp + 1 & \ell \\ mp & np + \Delta \end{pmatrix} \in \mathcal{H}_p$ (still ignoring $\alpha$). Observe that $q(w_*) \in \mathcal{B}_p$ if and only if $q(\beta_* w_*) \in \mathcal{B}_p$, so we may assume that $LL(w') \neq \beta$. If $LL(w') = \gamma$, then by Claim \ref{claim:images}, $q(w_*)(\infty) \in I_\gamma^-$. Observe that $q(w_*)(\infty) = \infty$ if and only if $q(w_*) = \begin{pmatrix} kp + 1 & \ell \\ 0 & np + \Delta \end{pmatrix}$, which is evidently in $\mathcal{B}_p$. On the other hand, if $q(w_*)(\infty) \neq \infty$, then since $q(w_*)(\infty) \in I_\gamma^-$, we have $|\frac{kp+1}{mp} | = |q(w_*)(\infty)| \geq \frac{1}{2}$, from which it follows that $|mp| \leq 2|kp+1|$. On the other hand, if $LL(w')  = \sigma$, then by Claim \ref{claim:images}, $q(w_*)(0) \in I_\sigma^-$. This means that $|\frac{\ell}{np+\Delta}| = |q(w_*)(0)| \leq \frac{2}{p}$, from which it follows that $p|\ell | \leq 2|np+\Delta|$. Thus, in all cases, $q(w_*) \in \mathcal{B}_p$.
\end{proof}

We can now prove Theorem \ref{thm:preciseimages}.
\begin{proof}
Theorem \ref{thm:subgroups}, Lemma \ref{lem:inequalities}, and Remark \ref{rem:inequalitytrivial} show that we have the inclusion $*(\mathcal{G}_p) \subseteq \set{ \pm \begin{pmatrix} B & D_B \\ 0 & (B^T)^{-1} \end{pmatrix} \, : \, B \in \mathcal{B}_p}$ for all $p$. It remains only to show the reverse containments. 

Let $A =  \pm \begin{pmatrix} kp+1 & \ell \\ mp & np + \Delta \end{pmatrix}$ be an element of $\mathcal{S}_p$. We will induct on the sum of the absolute values of the diagonal entries. Let $\kappa(A) = |kp + 1| + |np + \Delta|$. Suppose, first, that $\kappa(A) = 2$. Then $mp = 0$ or $\ell = 0$. Observe that $q((\beta_* \gamma_*)^\ell) = \begin{pmatrix} 1 & \ell \\ 0 & 1 \end{pmatrix}$, for any value of $p$. For $p = 2$, $q(\rho_*^2 \beta_* \gamma_* \rho_*)^m = \begin{pmatrix} 1 & 0 \\ 2m & 1 \end{pmatrix}$, for $p = 3$, $q(\beta_*\gamma_*\delta_*)^m = \begin{pmatrix} 1 & 0 \\ 3m & 1 \end{pmatrix}$, and for $p \geq 4$, $q(\beta_* \sigma_*)^m = \begin{pmatrix} 1 & 0 \\ mp & 1 \end{pmatrix}$. Further, if $X = \begin{pmatrix} a & b \\ c & d \end{pmatrix}$, then $\begin{pmatrix} a & b \\ -c & -d \end{pmatrix} = q(\beta_*)X$, $\begin{pmatrix} d & b \\ c & a \end{pmatrix} = \pm q(\beta_*)X^{-1}q(\beta^*)$, and $-X = q(\alpha_*)X$, or $-X = q(\rho_*^2)X$, so that we can generate any matrix $A$ with $\kappa(A) = 2$, regardless of the individual signs of the entries, or the value of $p$.

Now, suppose $\kappa(A) > 2$. In light of Remark \ref{rem:inequalitytrivial}, we may assume, regardless of $p$, that $|mp| \leq 2|kp+1|$ or that $p|\ell| \leq 2|np + \Delta|$. Since $\kappa(A) > 2$, we have $0 < |m|, |\ell|$. So, if $0 < |mp| < 2|kp+1|$, then multiplying $A$ on the left by $\begin{pmatrix} 1 & \pm1 \\ 0 & 1 \end{pmatrix}$, both of which are generated by $q(\beta_*)$ and $q(\gamma_*)$, as noted above, results in $\begin{pmatrix} kp+1\pm mp & \ell \pm (np + \Delta) \\ mp & np + \Delta \end{pmatrix}$, one of which will have a reduced value of $\kappa$. Similarly, if $0 < p|\ell| < 2|np + \Delta|$, then multiplying on the left by $\begin{pmatrix} 1 & 0 \\ \pm p & 1 \end{pmatrix}$, which can be generated in any of the $(q \circ *)(\mathcal{G}_p)$, as noted above, will result in $\begin{pmatrix} kp + 1 & \ell \\ mp \pm p(kp+1) & np + \Delta \pm p \ell \end{pmatrix}$, one of which will have a reduced value of $\kappa$. Thus, by the inductive hypothesis, all of $\pm \mathcal{B}_p$ is in the image of $(q \circ *)$.
\end{proof}

\section{Obstructions}
\label{section:Obstructions}
We now present obstructions to two given curves on the genus 2 Heegaard surface of $L(p, 1)$ being Goeritz equivalent.

\subsection{Homology Obstructions}
\label{subsection:HomologyObstructions}

Let $g_*$ be a matrix in $\mathcal S_p$ of the form \begin{equation} \tag{$g_*$} \label{eqn:gstar}  g_* = \omega \begin{pmatrix}  kp+1 & \ell & \varepsilon (k - \Delta n) & \varepsilon \Delta m \\
mp & np + \Delta & \varepsilon m & 0 \\ 0 & 0 & \Delta np + 1 & -\Delta  mp \\ 
0 & 0 & -\Delta \ell & \Delta(kp + 1) \end{pmatrix},\end{equation}

where $k, \ell, m, n \in \mathbb{Z}$, $\omega = \pm 1$, $\varepsilon = -1$ if $p = 3$, and $\varepsilon = 1$ otherwise, and $\Delta = \det \begin{pmatrix} kp+1 & \ell \\ mp & np + \Delta \end{pmatrix} = \pm 1$. Let $\textbf{v} = [K] = (a, x, b, y)^T$ and $\textbf{v}' = [K'] = (a', x', b', y')^T$ be vectors in $H_1(F; \mathbb{Z})$, and let $p$ be an integer with $p \ge 2$. Let \begin{equation} \tag{$s$} \label{eqn:s} s = ap+ \varepsilon b, \qquad s' = a'p + \varepsilon b',\end{equation}
 and let  
 \begin{equation} \tag{$\delta$} \label{eqn:delta} \delta = bs + xyp, \qquad \delta ' = b's' + x'y'p. \end{equation} 
 If $\delta, \delta' \neq 0$, we define 
 \begin{multline} \tag{$c{-}f$} \label{eqn:cdef} c = \frac{\Delta bx' - b'x}{\delta}, \hspace{.1in} d = \frac{bs' + \Delta xy'p}{\delta}, \hspace{.1in} e = \frac{ b's + \Delta x' y p}{\delta}, \mbox{ and } \hspace{.05in} f= \frac{ys' - \Delta y's}{\delta}, \\ c' = \frac{\Delta bx' - b'x}{\delta'}, \hspace{.1in} d' = \frac{bs' + \Delta xy'p}{\delta'}, \hspace{.1in}  e' = \frac{b's+ \Delta x' y p }{\delta'}, \mbox{ and } \hspace{.05in}  f'= \frac{ys' - \Delta y's}{\delta'}.\end{multline}

\begin{lem}[The Determinant Condition]
If $g_* \in \mathcal S_p$, then $$knp + \Delta k + n -\ell m = 0.$$
\end{lem}

\begin{proof}
Because the upper left block of the matrix $g_*$ has determinant $\Delta = \pm 1$, we know that for any $g_*$, $(kp+1)(np+\Delta) - \ell mp = \Delta$. This gives us the restriction on $k,\ell,m$, and $n$ that $knp + \Delta k + n -\ell m = 0$. \end{proof}

\begin{lem}\label{lem:variablefacts}
Let ${\bf{v}} = [K] = (a, x, b, y)^T$ and ${\bf{v}}'= [K'] = (a', x', b', y')^T$ be vectors in $H_1(F; \mathbb{Z})$, let $p$ be an integer with $p \ge 2$, and suppose $\delta$ and $\delta'$ as defined above are nonzero. Then we have the following relationships. 

\begin{enumerate}

\item\label{lem:cfpminusde}  $cfp - de = - \delta' /\delta$ and $ c'  f' p -  d'  e' = - \delta/ \delta'$.

\item\label{lem:cdefvsklmn} If there is a matrix $g_* \in \mathcal S_p$ so that $g_* {\bf v} = {\bf v}'$
, then $c=  c' = \Delta \omega  m$, $d = d' = \omega( kp + 1)$, $e=  e'  = \omega(\Delta np + 1)$, $f =  f'  =  \omega \ell$. 

\item\label{lem:deltavsdeltaprime} If there is a matrix $g_* \in \mathcal S_p$ so that $g_* {\bf v} = {\bf v}'$
, then $\delta = \delta'$.
\end{enumerate}
\end{lem}

\begin{proof}
The proof that $cfp - de = - \delta' /\delta$ requires only algebraic computation, given the definitions of the appropriate variables. By the definition of the variables, $\delta'  c'  = \delta c$,  $\delta'  d'  = \delta d$,  $\delta'  e' = \delta e$,  and $\delta' f'  = \delta f$, so  $c'  f' p -  d'  e'  = \frac{\delta ^2}{(\delta')^2}(cfp - de) = - \delta/\delta'$.

For Part (\ref{lem:cdefvsklmn}), suppose that such a $g_*$ exists. Then we have
\[\begin{matrix} a' & = &  \omega \left[ (kp + 1) a + \ell x + \varepsilon (k - \Delta n) b + \Delta \varepsilon m y\right], \\
x' & = &  \omega \left[ mpa + (np +\Delta) x +  \varepsilon m b\right], \hspace{.75in}\\
b' &=& \omega \left[  (\Delta np + 1)b - \Delta m p y\right], \mbox{and} \hspace{.7in}\\
y' &=&   \omega \left[ -\Delta \ell b + \Delta(kp + 1) y\right].\hspace{1.0in}
\end{matrix}\]
Using these primed variables with Equations (\ref{eqn:cdef}) gives the desired result.

To prove the same with $c', d', e',$ and $f'$, we also use algebraic manipulations, noting that  $(g_*)^{-1}$ is 
\[(g_*)^{-1} = \omega\begin{pmatrix}  \Delta np+1 & - \Delta \ell & \varepsilon ( \Delta n - k) &  - \varepsilon m \\
- \Delta mp & \Delta (kp + 1) & - \Delta \varepsilon m & 0 \\ 0 & 0 &kp + 1 & mp \\ 
0 & 0 &  \ell & np + \Delta \end{pmatrix},\]
so that if $(g_*)^{-1} \textbf v' = \textbf v$, then we can write $a, x, b,$ and $y$ as 
\[\begin{matrix} a & = &  \omega \left[ (\Delta np+1 )a' - \Delta \ell x' + \varepsilon ( \Delta n - k) b'   - \varepsilon m y' \right],\\
x & = & \omega \left[  - \Delta mp a' + \Delta (kp + 1) x' - \Delta \varepsilon mb' \right], \hspace{.55in}\\
b &=&  \omega \left[ (kp + 1)b' + m p y'\right], \mbox{and} \hspace{1.22in}\\
y &=&   \omega \left[ \ell b' + (np + \Delta) y'\right]. \hspace{1.5in}
\end{matrix}\]

Part (\ref{lem:deltavsdeltaprime}) of the lemma is a consequence of part (\ref{lem:cdefvsklmn}). As $d = d' = \omega(kp + 1)$ is never zero, the denominators of $d$ and $d'$ must agree.
\end{proof}

\begin{thm} \label{thm:linalgobstruction} Suppose there exists a Goeritz group element $g \in \mathcal{G}_p$ carrying the curve $K$ to $K'$, with ${\bf{v}} = [K] = (a, x, b, y)^T$ and ${\bf{v}}' = [K'] = (a', x', b', y')^T$. 
Then, for $g_*$, $\delta$ and $\delta'$, and $s$ and $s'$ as in Equations (\ref{eqn:gstar}), (\ref{eqn:delta}), and (\ref{eqn:s}), respectively,
\begin{enumerate}
\item\label{cond:bcong} $b' \equiv \omega b \, (\mbox{mod }p)$,
\item\label{cond:gcd} $\gcd(b, y) = \gcd(b', y')$,
\item\label{cond:delta} $\delta$ and $\delta'$ divide each of $\Delta bx' - b'x$, $bs' + \Delta xy'p$, $x' y p + b's$, and $ys' - \Delta y's$, and
\item\label{cond:ineqcond} $|(\Delta b x' - b' x)p | \leq 2|bs' + \Delta x y' p| \mbox{ or } p| ys' - \Delta y' s |  \leq  2 | b's + \Delta x'yp |.$
\end{enumerate}
\end{thm}

\begin{proof}
If such a matrix exists, then we have $b' = \omega(\Delta np + 1)b - \Delta \omega m p y$ and $y ' = - \Delta \omega \ell b + \Delta \omega (kp + 1) y$. Reducing the former equation modulo $p$ gives $b' \equiv \omega b \, (\mbox{mod }p)$.

 Further, the two equations together show that $\gcd(b, y)$ divides both $b' $ and $y'$, so  $\gcd(b,y)$ divides $\gcd(b',y')$. Because we must also have $(g_*)^{-1} \textbf{v} ' = \textbf{v}$,  we see $b = \omega (kp+1) b' + \omega mp y'$ and $y = \omega\ell b' + \omega(np + \Delta) y'$. Then $\gcd(b',y')$ divides $\gcd (b,y)$, and so the two greatest common divisors are equal.

Part (\ref{lem:cdefvsklmn}) of Lemma \ref{lem:variablefacts} shows that if such a $g_*$ exists then each of $c, c', d,  d', e, e', f$, and $f'$ is an integer. Then, by the definitions of these variables, $\delta$ and $\delta'$ divide the numerator of each fraction, so the divisibility condition is satisfied.

To prove the final condition, we rely on Part (\ref{lem:cdefvsklmn}) of Lemma \ref{lem:variablefacts}, which says that if an element of $\mathcal S_p$ satisfies  $g_* \textbf v = \textbf v'$
, then $c=  c' = \Delta \omega m$, $d = d' = \omega (kp + 1)$, $e= e'  =\omega (\Delta np + 1)$, $f = f'  =  \omega \ell$. Theorem \ref{thm:preciseimages} indicates the condition that maintains membership in $*(\mathcal G_p)$ is $|mp | \leq 2|kp + 1|$ or $p| \ell |  \leq  2 | np + \Delta |$. Since $m = \Delta \omega c$ and $\ell = \omega f$, we have $|m| = |c|$ and $|\ell| = |f|$. Further $d =\omega( kp + 1)$ and $e = \Delta\omega (np + \Delta)$, so $|d| = |kp + 1|$ and $|e| = |np + \Delta |$. Then the inequality $|mp | \leq 2|kp + 1|$ is equivalent to $|cp | \leq 2|d| $, and the inequality $p| \ell |  \leq  2 | np + \Delta |$ is equivalent to $p| f |  \leq  2 | e |$. Since each of $c, d, e, f$ has $\delta$ in the denominator, we can multiply each of these inequalities by $|\delta|$ to obtain the desired condition.
\end{proof}

Next, we give some explicit conditions, at least one of which must be satisfied, if two curves in the genus 2 Heegaard surface for $L(p,1)$ are homologically Goeritz equivalent. 

\begin{thm}\label{thm:vectorlist} Let $c, d, e, f, \delta, \delta'$, and $s$ be defined as above with $\delta$ and $\delta'$ nonzero.
Given two vectors ${\bf v}  = (a, x, b, y)^T$ and ${\bf{v}}' = (a',x',b',y')^T$, if there is a matrix $g_*$ induced by an element of the Goeritz group so that $g_* {\bf v}= {\bf v}'$, then $(a,x,b,y)$ and $(a', x', b', y')$ satisfy at least one of the following:

\begin{enumerate}
\item\label{list:bg}  $(a', x', b', y')^T =\omega ( a, \Delta x, b, \Delta y)^T$,
\item\label{list:pb}  $(a', x', b', y')^T = \omega(a + \ell x, \Delta x, b, \Delta (y - \ell b))^T$ for some integer $\ell \neq 0$,
\item\label{list:lav}  $(a', x', b', y')^T = \omega(a + \Delta \varepsilon my, \Delta x + ms, b - \Delta mpy, \Delta y)^T$ for some integer $m \neq 0$,
\item\label{list:gray}  $(a', x', b', y')^T = (\omega(a + cf s)+f x  + \varepsilon c y, \Delta (\omega x +   c s),  \omega b - c p y , \Delta( \omega(cfp + 1) y-fb))^T$,
\item\label{list:or1}  $(a', x', b', y')^T = (-\omega a + fx + cy, -\Delta ( \omega x - c s), -\omega b - 2cy, -\Delta (\omega y + fb) )^T $ with $p = 2$ and at least one of $c$ or $f$ is zero,
\item\label{list:or2}  $(a', x', b', y')^T =  \left(da + fx + \varepsilon\left(\frac{d-e}{p}\right)b + \varepsilon cy, \Delta (ex + cs), eb - cp y, \Delta(dy - fb)\right)^T$ and  $|cp | \leq 2|d| $ or $p| f |  \leq  2 | e |$, or
\item\label{list:bb}  $(a', x', b', y')^T = (\omega (a + x) + \varepsilon c (y-b), \Delta ((cp + \omega) x + c s) , (cp+\omega) b - cp y, \Delta\omega (y-b))^T$. \end{enumerate}
\end{thm}

\begin{rem}
In the study of lens space surgeries, one might be interested in the knots that lie almost entirely in the genus 1 Heegaard surface for $L(p,1)$, except for one arc that extends out of the surface. One can stabilize the genus 1 Heegaard surface to a genus 2 Heegaard surface by adding a 1-handle along this arc, producing a curve $C$ in the genus 2 Heegaard surface $F$ for $L(p,1)$ with homology vector in $H_1(F; \mathbb Z)$ of the form $(a, 0, b, \pm 1)$. In this case, the list provided in Theorem \ref{thm:vectorlist} gives a short list of vectors that are homologically Goeritz equivalent to $C$. An analysis of these vectors, such as the one provided in \cite{DolRatG2GES3}, can provide potential candidates of curves that are topologically Goeritz equivalent to $C$.
\end{rem}

\begin{proof}
Suppose that such an element of the Goeritz group exists, i.e., that $g_* \textbf{v} = \textbf{v}'$. Because $g_*$ is induced by a homeomorphism, it must be an invertible matrix with the property that $g_*^{-1}\textbf{v}' = \textbf{v}$. We reframe these matrix equations into linear systems of equations represented instead by the matrix equations $V \textbf{g} = \textbf{v} + \omega  M \textbf{v}'$ and $V' \textbf{g} = \textbf{v}' + \omega  M \textbf{v}$, where $\omega  = \pm 1$, $\textbf{g} = (k,\ell,m,n)^T$, $M$ is the diagonal matrix with diagonal entries $-1$, $- \Delta$, $-1$, and $-\Delta$, and $V$ and $V'$ are given by 
\[V =  \omega  \begin{pmatrix}  ap+\varepsilon b & x & \varepsilon \Delta y & -\varepsilon \Delta b \\
0 & 0 & ap+\varepsilon b & xp \\ 0 & 0 & -\Delta yp  & -\Delta  bp \\ 
 -\Delta yp& -\Delta b & 0 & 0 \end{pmatrix}\]
and 
\[V' = \omega  \begin{pmatrix}  -\varepsilon b' & - \Delta x' & -\varepsilon y' & \Delta (a' p + \varepsilon b') \\
\Delta x'p & 0 & -\Delta(a'p+\varepsilon b') & 0 \\ b'p & 0 & y'p  & 0 \\ 
0 & b' & 0 & y'p \end{pmatrix}.\] We note that $\textbf{v} + \omega  M \textbf{v}' =   \omega  M(\textbf{v}' + \omega  M \textbf{v})$, so we have the equation $V \textbf{g} = \omega  MV' \textbf{g}$. Then, when $V$ is invertible we have that $V^{-1}M V' \textbf{g} = \omega  \textbf{g}$, i.e., $\textbf{g}$ is eigenvector for $V^{-1}M V'$ with eigenvalue $\omega  = \pm 1$.

From here, for compactness of notation, we use Equations (\ref{eqn:s}) and (\ref{eqn:delta}). We find that $\det V = -p\delta^2$ and $\det V' = \Delta p (\delta ')^2$, so these matrices are invertible over $\mathbb{Q}$ exactly when $\delta$ and $\delta'$, respectively, are nonzero. Since the determinants of $V$ and $V'$ are integers and $p \ge 2$, there are no values of $\delta$ and $\delta '$ that provide invertibility of $V$ and $V'$ over $\mathbb Z$. It is, nevertheless, possible to find integral solutions for $\textbf{g}$ in some cases.

Because Lemma \ref{lem:variablefacts} tells us that when such a $g_*$ exists, $\delta = \delta'$, we use $\delta$ wherever either $\delta$ or $\delta'$ appears. 

Since we assume $\delta \neq 0$, $V^{-1}$ is invertible over $\mathbb{Q}$. We compute $V^{-1}$ to be 
\[V^{-1} =  \frac{\omega }{\delta} \begin{pmatrix} b & 0 & \frac{\varepsilon b}{p} & \Delta x\\
yp& 0 & \varepsilon y & -\Delta s\\ 0 & b & -\Delta x  & 0 \\ 
0 & y &  \frac{\Delta s}{p} & 0 \end{pmatrix}.\]

Then the matrix $V^{-1} M V' $ is 
\[V^{-1}M V' =  \frac{1}{\delta} \begin{pmatrix} 0 & \Delta b x' - b' x & 0 &  -\Delta( b s' + \Delta xy' p) \\
 0& b' s + \Delta x' y p & 0 & -\Delta p ( y s' - \Delta y' s)\\ 
-\Delta p( \Delta bx' -  b' x) & 0 & bs' + \Delta xy' p  & 0 \\ 
-\Delta( b's + \Delta x' y p) & 0 & ys' - \Delta y' s & 0 \end{pmatrix}\]

Next, we compute the eigenspace for $V^{-1} M V' $ associated to $\omega $. Once again, we make the notation more compact by using Equations (\ref{eqn:cdef}). Then we can write $V^{-1} M V'  - \omega  I$  as 
\[V^{-1}M V' - \omega  I =  \begin{pmatrix} -\omega  & c & 0 &  -\Delta d \\
 0& e - \omega  & 0 & -\Delta f p\\ 
-\Delta cp & 0 & d - \omega   & 0 \\ 
-\Delta e & 0 & f &  -\omega  \end{pmatrix}.\]

In order to row reduce this matrix over $\mathbb{Q}$, we must consider several cases. We split these cases into two main branches: those where $e - \omega  = 0$ (Cases \ref{case:brightgreen} - \ref{case:navy}), and those where $e-\omega  \neq 0$ (Cases \ref{case:kellygreen} - \ref{case:beige}). 

In the cases where $e - \omega  = 0$, we know that $e = \omega $, so we can simplify this matrix, via row reduction, to
\begin{align} \tag{A}\label{matrix:A} \begin{pmatrix} -\omega  & c & 0 &  -\Delta d \\
 0& 0 & 0 & -\Delta f p\\ 
-\Delta cp & 0 & d - \omega   & 0 \\ 
0 & -\Delta c & f &  d-\omega  \end{pmatrix}.\end{align}

We consider the following cases:

\begin{enumerate}
\item\label{case:brightgreen} $e - \omega = 0; \, c = d - \omega  = f = 0$

In this case, the matrix in (\ref{matrix:A}), reduces to 

\[ \begin{pmatrix} -\omega  & 0 & 0 &  -\Delta \omega  \\
 0& 0 & 0 & 0\\ 
 0 & 0 & 0  & 0 \\ 
0 & 0 & 0 &  0\end{pmatrix},\]
so the eigenspace contains rational vectors of the form $(t, u, z, -\Delta t)^T$, where $t$, $u$, and $z$ are free variables. Consider $k,\ell,m$, and $n$ to be integers so that the vector $ (k,\ell,m,n)^T$ is of the form $(t, u, z, -\Delta t)^T$. Then $n = -\Delta k$, and the determinant condition requires $-\Delta k^2 p = \ell m $. Using the equation $g_* \textbf{v} = \textbf{v}'$,  we find that $\omega a ' = ( kp+1)a +  \ell x + 2 \varepsilon   k b + \Delta \varepsilon   m y $, $\omega x' =   mpa + \Delta  (1-  kp) x + \varepsilon   mb$, $\omega b' = (1- kb)b - \Delta  m p y$, and $\omega y' = - \Delta  \ell b + \Delta(  kp+1) y $. Then, using Lemma \ref{lem:variablefacts}(\ref{lem:cdefvsklmn}), we find that $c = \Delta\omega  m$, so the condition that $c = 0$ requires $m = 0$. Since the determinant condition is $-\Delta k^2 p = \ell m $ in this case, we must also have $k = 0$. Further, $f =  \omega \ell $, and the condition that $f = 0$ requires $\ell = 0$ as well. Finally, we find that $e = \omega$, so the condition that $e - \omega  = 0$ is satisfied. Hence, the only possible integral vector $\textbf{g}$ that satisfies $V^{-1}M V' \textbf{g} = \textbf{g}$ with the necessary conditions is the zero vector, and $(a', x', b', y')^T = \omega (a, \Delta x, b, \Delta y)^T$, listed in part (\ref{list:bg}) of the theorem. 

\item\label{case:plainblue} $e - \omega = 0; \, c = d - \omega   = 0$, $f \neq 0$ 

With these assumptions, we have the matrix in (\ref{matrix:A}) is
\[ \begin{pmatrix} -\omega  & 0 & 0 &  -\Delta \omega  \\
 0& 0 & 0 & -\Delta f p\\ 
0& 0 & 0  & 0 \\ 
0 & 0 & f &  0\end{pmatrix}.\]

This matrix row reduces to a matrix $N$ where $n_{1, 1} = n_{2, 3} = n_{3, 4} = 1$ and all other entries are 0. Then the null space of this matrix consists of vectors of the form $(0,u,0,0)^T$, where $u \in \mathbb Q$. Then the integral vectors $(k,\ell,m,n)^T$ of this result in the matrix
\[ g_* = \omega\begin{pmatrix}  1 & \ell & 0 &0 \\
0 &  \Delta & 0 & 0 \\ 
0 & 0 &  1 & 0 \\ 
0 & 0 & -\Delta \ell & \Delta \end{pmatrix}.\]

Here, the determinant condition is satisfied, and we find $g_* \textbf{v}$ to be  $(a',x',b',y')^T = \omega (a + \ell x, \Delta x, b, \Delta (y - \ell b))^T$. Using these new definitions for $a', x', b',$ and $y'$, we find that $c = 0$, $d - \omega  = 0$, and $e - \omega  = 0$  are all satisfied. Further $f = \omega \ell$, so the condition that $f \neq 0 $ is equivalent to requiring $\ell \neq 0$. Because $c = 0$ and $|d| = 1$, the inequality $|cp | \leq 2|d| $ always holds. Then we have the examples $(a',x',b',y')^T = \omega (a + \ell x, \Delta x, b, \Delta (y - \ell b))^T$ for some nonzero integer $\ell$  listed in part (\ref{list:pb}) of the theorem. 

\item\label{case:violet} $e - \omega = 0; \, c = 0$, $d - \omega  \neq 0$

Note here that no condition on $f$ needs to be initially assumed. Lemma \ref{lem:variablefacts} tells us that when an integer matrix $g_*$ exists, $e = \omega(\Delta np + 1)$ and $d = \omega (kp+1)$. The requirement that $e = \omega$ then forces $n = 0$, and the determinant condition becomes $\Delta k = \ell m$. Since $d\neq \omega$, $k\neq 0$ and both of $\ell = \omega f$ and $m = \Delta \omega c$ are nonzero. Hence $c$ and $f$ must both be nonzero, and no integer solutions can arise from this case.

\item\label{case:lavender} $e - \omega = 0; \, d - \omega  =  f = 0$, $c \neq 0$

With these assumptions, we have the matrix in (\ref{matrix:A}) equal to
\[ \begin{pmatrix} -\omega  & 0 & 0 &  -\Delta \omega  \\
 0& -\Delta c & 0 & 0\\ 
0& 0 & 0  & cp \\ 
0 & 0 & 0 &  0\end{pmatrix}.\]

This matrix row reduces to a matrix $N$ with $n_{1, 1} = n_{2, 2} = n_{3, 4} = 1$ and all other entries 0. Then the eigenspace consists of vectors of the form $(0,0,z,0)^T$, where $z \in \mathbb Q$. Then the integral vectors $(k,\ell,m,n)^T$ of this form result in the matrix  
\[g_* = \omega\begin{pmatrix}  1 & 0 & 0 & \varepsilon \Delta m \\
 mp &  \Delta & \varepsilon  m & 0 \\ 
0 & 0 &  1 & -\Delta mp \\ 
0 & 0 & 0 & \Delta \end{pmatrix},\]
and $(a',x',b',y')^T = \omega(a + \Delta \varepsilon  my, mpa + \Delta x + \varepsilon mb, b - \Delta  mpy, \Delta y)^T$. Given these expressions for the primed variables in terms of the unprimed ones, $f = 0$ and $d = e = \omega$. Further, $c = \Delta\omega m$, so the condition that $c \neq 0$ is equivalent to the condition that $m \neq 0$. The determinant condition holds as well, and we have the examples $(a',x',b',y')^T = \omega(a + \Delta \varepsilon my, mpa + \Delta x + \varepsilon mb, b - \Delta  mpy, \Delta y)^T$ with $m\neq0$, listed in part (\ref{list:lav}) of the theorem.

\item $e - \omega = 0; \, f = 0$, $c \neq 0$, $d - \omega   \neq 0$

For the same reason as in Case \ref{case:violet}, no examples can arise from this case.

\item\label{case:navy} $e - \omega = 0; \, c \neq 0$, $f \neq 0$ 

Note here that no condition on $d - \omega$ needs to be initially assumed. The matrix in (\ref{matrix:A}) now row reduces to a matrix with $(3, 4)$ entry equal to $d - \omega   - \omega  cfp$. Lemma \ref{lem:variablefacts} tells us that, when an integer matrix $g_*$ exists, $cfp - de = -1$. Since $e = \omega$ this equation becomes $cfp - d \omega= -1$. Multiplying through by $\omega$ and rearranging, we have $d - \omega - \omega cfp = 0$.

The matrix in (\ref{matrix:A}) then further row reduces to 

\[  \begin{pmatrix} -\omega  & 0 & \Delta f & 0 \\
 0&  -\Delta c & f & 0 \\ 
0& 0 & 0  & 1 \\ 
0 &0 &0 & 0 \end{pmatrix}.\]

The null space of this matrix consists of rational multiples of the vector $(\Delta \omega  f, \frac{\Delta f}{c}, 1, 0)^T$. In order to create an integer vector that provides a solution to our equation, we multiply this vector by some nonzero integer $z$ so that all entries become integers. Then we have the integer vector $\textbf{g} = (\Delta \omega  f z, \frac{\Delta f z}{c}, z, 0)^T$.
 Lemma \ref{lem:variablefacts}(\ref{lem:cdefvsklmn}) tells us that the third entry of $\textbf{g}$, $z$, is equal to $\Delta \omega c$. Further, $e = \omega(\Delta np + 1)$, where $n=0$, so $e - \omega = 0$ is satisfied. The determinant condition holds and the condition $d - \omega   - \omega  cfp  = 0$ is satisfied because $d = \omega (kp+1)$. Then the vector becomes $( cf,  \omega f, \Delta \omega c, 0)^T$, and we have the matrix

\[g_* = \omega \begin{pmatrix}  cfp+1 & \omega f & \varepsilon cf  & \varepsilon\omega c \\
\Delta \omega cp &  \Delta & \Delta \varepsilon\omega c & 0 \\ 
0 & 0 &  1 & - \omega cp \\ 
0 & 0 & -\Delta\omega f & \Delta(cfp + 1) \end{pmatrix},\]
and $a' = \omega (cfp+1)a + f x + \varepsilon \omega cf b + \varepsilon c y$, 
$x' = \Delta (cp a +  \omega x + \varepsilon c b )$, 
$b' = \omega b - c p y$, and $
y' = \Delta(  \omega (cfp + 1) y - f b).$ We use $s = ap + \varepsilon b$ to rewrite this vector as shown in part (\ref{list:gray}) from the theorem.

\end{enumerate}

This concludes the cases where $e - \omega  = 0$, and from here, we assume that $e-\omega  $ is nonzero. We find that our original version of $V^{-1}MV - \omega  I$ row reduces to  

\[  \begin{pmatrix} -\omega  & 0 & 0 &  \frac{\Delta (cfp - de + \omega  d)}{e - \omega } \\
 0& e - \omega  & 0 & -\Delta f p\\ 
0 & 0 & d - \omega   & -\frac{ \omega  cp (cfp - de + \omega  d)}{e - \omega } \\ 
0 & 0 & f & - \frac{\omega (e-\omega  + e(cfp - de + \omega  d)) }{e-\omega }\end{pmatrix}.\]

Using Lemma \ref{lem:variablefacts}(\ref{lem:cfpminusde}) with $\delta = \delta'$, we can rewrite this matrix as 

\begin{align} \tag{B} \label{matrix:B} \begin{pmatrix} -\omega  & 0 & 0 &  \frac{\Delta \omega ( d -\omega )}{e - \omega } \\
 0& e - \omega  & 0 & -\Delta f p\\ 
0 & 0 & d - \omega   & -\frac{ cp ( d -\omega )}{e - \omega } \\ 
0 & 0 & f &  \frac{1-de }{e-\omega}\end{pmatrix}.\end{align}

Since we know this matrix has rank at least two, we consider the case that the third column is not a pivot column (Case \ref{case:kellygreen}) and the cases that the third column is a pivot column (the remainder of the cases).

\begin{enumerate}\setcounter{enumi}{6}

\item\label{case:kellygreen} $e - \omega  \neq 0$; $d - \omega  = f =  0$

Similarly to Case \ref{case:violet}, $d- \omega = 0$ implies that $k = 0$. Then the determinant condition and the condition that $e - \omega  \neq 0$ require that $c$ and $ f$ are both nonzero. Hence no integer solutions arise from this case.

\end{enumerate}

When $d- \omega  \neq 0$, the matrix in (\ref{matrix:B}) row reduces so that the (4,4) entry is  $$ \frac{1-de }{e-\omega}- \frac{f}{d-\omega }\cdot \frac{ -cp ( d -\omega )}{e - \omega } , $$ which is identically 0. 

\begin{enumerate}\setcounter{enumi}{7}
\item\label{case:rainbow} $e-\omega  \neq 0$; $d - \omega  \neq 0$ 

Note here that no condition on $f$ needs to be initially assumed. The matrix in (\ref{matrix:B}) row reduces to

\[ \begin{pmatrix} -\omega  & 0 & 0 &  \frac{\Delta \omega ( d -\omega )}{e - \omega } \\
 0& e - \omega  & 0 & -\Delta f p\\ 
0 & 0 & d - \omega   & -\frac{ cp ( d -\omega )}{e - \omega } \\ 
0 & 0 & 0 & 0\end{pmatrix}.\]
The rational vectors in the null space of this matrix are of the form $$\left( \frac{\Delta (d-\omega) }{e - \omega}, \frac{\Delta f p }{e-\omega }, \frac{ cp }{e - \omega }, 1\right)^T, $$
We can multiply this vector by some nonzero integer $z$ so that it becomes the integer vector: $$\left( \frac{\Delta (d-\omega) z}{e - \omega}, \frac{\Delta f p z}{e-\omega }, \frac{ cp z}{e - \omega }, z\right)^T. $$

From here, we consider two sub-cases:  at least one of $c$ or $f$ is zero, or neither $c$ nor $f$ is zero.

\begin{enumerate}
\item  Suppose first that at least one of $c$ or $f$ is zero. By Lemma \ref{lem:variablefacts}, we have $cfp - de = -de = -1$. Since $d$ and $e$ are both integers, we have $d = e = \pm 1$. Because $d - \omega , e-\omega  \neq 0$, this means, $d = e = -\omega $, and we have the vector $$\left(\Delta  z, -\frac{\Delta \omega  fpz}{2}, -\frac{\omega  cpz}{2}, z\right)^T.$$ 

Since one of $c$ or $f$ is zero, the determinant condition becomes $\Delta z^2 p + 2z = 0$. That is $\Delta z p = -2$. Then $p = 2$ and $z = - \Delta$. We have the matrix

\[g_* = \omega \begin{pmatrix}  -1 & \omega  f & 0 & \omega  c \\
2\Delta \omega  c & -\Delta & \Delta \omega  c & 0 \\ 
0 & 0 & - 1 & -2 \omega  c \\ 
0 & 0 & -\Delta\omega  f & -\Delta \end{pmatrix},\]

where at least one of $c$ or $f$ is zero. 
Then we have $a' = -\omega a +  f x +c y$, $x' = 2\Delta  c a - \Delta \omega x + \Delta c b$, $b' = -  \omega b - 2 c y$ and $y' = -\Delta   f b - \Delta  \omega y$, where at least one of $c$ or $f$ is zero. These give us the examples listed in part (\ref{list:or1}) of the theorem.

\item Now suppose that neither $c$ nor $f$ is zero. By Lemma \ref{lem:variablefacts}, $d = \omega kp + \omega$ and $e = \Delta \omega np + \omega$. Then $k = \frac{\omega (d-\omega)}{p}$ and $n = \frac{\Delta \omega (e - \omega)}{p}$, which are both integers. Further, $c = \Delta  \omega m$ and $f = \omega  \ell$, so $c = \frac{\Delta \omega  cp z}{e - \omega }$ and $f = \frac{\Delta \omega fpz}{e-\omega }$. These equations tell us $z = \frac{\Delta\omega (e-\omega )}{p}$, and our vector becomes $$\left(\frac{\Delta  (d- \omega  )}{p}, \omega f, \Delta\omega c, \frac{\Delta \omega(e - \omega)}{p}\right)^T.$$

The determinant condition tells us $$\frac{\omega(d-\omega)(e - \omega)}{p^2} \cdot p + \frac{d-\omega}{p} + \frac{\Delta \omega(e-\omega)}{p} -\Delta cf = 0.$$ Multiplying this equation by $p$ gives $$\omega(d-\omega)(e - \omega) + d-\omega+ \Delta \omega(e-\omega)-\Delta cfp = 0.$$ Since $\delta' = \delta$ and $cfp - de = -1$,  we can use $cfp = de -1$ to obtain $$\omega(d-\omega)(e - \omega) + d-\omega+ \Delta \omega(e-\omega)-\Delta (de -1) = 0.$$ This equation reduces to $$(\omega-\Delta ) de   +( \Delta \omega-1) e  = 0.$$
We can divide through by $e$, since $e = 0$ would imply that $p$ divides $\omega$, and we have $$(\omega-\Delta ) d  = \omega(\omega - \Delta).$$
Since $d-\omega \neq 0$, $\omega = \Delta$.

Then we have the matrix 
\[g_* = \Delta \begin{pmatrix}  \Delta d & \Delta f & \Delta \varepsilon \left(\frac{d-e}{p}\right) & \Delta \varepsilon c \\
cp &  e &  \varepsilon  c & 0 \\ 
0 & 0 & \Delta e & -\Delta cp \\ 
0 & 0 & - f &  d \end{pmatrix},\]

and we have our next set of examples, listed in part (\ref{list:or2}) of the theorem, provided that the condition $|cp | \leq 2|d| $ or $p| f |  \leq  2 | e |$ is met, as required by Theorem \ref{thm:linalgobstruction}.

\end{enumerate}
\end{enumerate}

Finally, we consider the conditions $d-\omega  = 0$ and $f \neq 0$, so that the matrix in (\ref{matrix:B}) row reduces to
\[ \begin{pmatrix} -\omega  & 0 & 0 & 0 \\
 0& e - \omega  & 0 & -\Delta f p\\ 
0 & 0 & 0  & 0 \\ 
0 & 0 & f & - \omega\end{pmatrix}.\]

\begin{enumerate}\setcounter{enumi}{8}

\item\label{case:beige} $e-\omega  \neq 0$; $d - \omega  = 0$, $f \neq 0$ 

The null space of this matrix consists of rational multiples of the vector $$\left(0, \frac{\Delta f p}{e- \omega },1, \frac{\omega }{f}\right)^T.$$ We choose some nonzero integer $z$ that produces an integer vector
$$\left(0, \frac{\Delta f pz}{e- \omega },z, \frac{\omega  z}{f}\right)^T.$$

Now Lemma \ref{lem:variablefacts}(\ref{lem:cdefvsklmn}) tells us that $e = \omega(\Delta np + 1)$, where $n$ is the fourth entry of our integer vector, so $e - \omega  = \Delta \omega np = \frac{\Delta  p z}{f}$. Further $f = \omega \ell =   \frac{\Delta \omega f pz}{e- \omega } = \omega f^2$. Since $f \neq 0$, we find that $f = \omega$ and $z = \frac{\Delta \omega (e-\omega)}{p}$, and the vector reduces to  
$\left(0, 1,z, z\right)^T.$

Using the equation $c = \Delta \omega m = \Delta \omega z$, so that $z = \Delta \omega c$, we have the integer vector $$(0,1, \Delta\omega c, \Delta\omega c)^T.$$ 
This vector satisfies the determinant condition, so we have the matrix \[g_* =  \omega \begin{pmatrix}  1 & 1 & -\varepsilon \omega  c & \varepsilon \omega c \\
\Delta \omega cp & \Delta \omega cp + \Delta &  \Delta \varepsilon\omega  c & 0 \\ 
0 & 0 & \omega cp + 1 & -\omega cp \\ 
0 & 0 & -\Delta  & \Delta \end{pmatrix},\]
and $a' = \omega (a + x) + \varepsilon c (y-b)$, $x ' = \Delta cp a + \Delta (cp + \omega ) x + \Delta \varepsilon c b $, $b' = (cp+\omega ) b - cp y$, and $y' = \Delta\omega (y-b)$. These examples are listed in part (\ref{list:bb}) of the theorem.

\end{enumerate}

This concludes our case work, and we have proven the theorem.
\end{proof}

\subsection{Homotopy Obstruction}
\label{subsection:HomotopyObstruction}

Using Theorem \ref{thm:kernel}, we can also prove an analogue of Theorem 1.3 from \cite{DolRatG2GES3}.

\begin{thm} 
\label{thm:HomotopyObstruction} Let $(F, V, W)$ be the genus two Heegard splitting of $L(p, 1)$ for $p \geq 2$. Suppose $K$ and $K'$ are homologous simple closed curves on $F$. If $K$ is not freely homotopic to $K'$ in $V$ or if $K$ is not freely homotopic to $K'$ in $W$, then $K$ and $K'$ are not Goeritz equivalent.
\end{thm}

\begin{proof}

If an essential (separating) curve on $F$ bounds a disk in each handlebody, then the union of these disks is a reducing sphere for the Heegaard splitting. We will call such a curve a reducing sphere curve, and a Dehn twist around such a curve a reducing sphere twist. We claim that $K$ and $K'$ are Goeritz equivalent if and only if they are related by a sequence of reducing sphere twists.  

Observe that $\beta^2$ is a Dehn twist around the `belt curve,' the curve that separates $\set{a, b}$ from $\set{x, y}$, which is a reducing sphere curve. Conjugating $\beta^2$ by elements of the Goeritz group will result in another reducing sphere twist. As in $S^3$, the Goeritz group acts transitively on reducing spheres. To see this, note that for genus two Heegaard splittings of lens spaces, Cho and Koda \cite{ChoKodHSG2HS} prove there is a one-to-one correspondence between dual pairs and reducing spheres. In \cite{ChoG2GGLS}, Cho proves that the Goeritz group acts transitively on primitive disks, and that the stabilizer of a primitive disk acts transitively on its dual disks. Thus, given any two reducing spheres, $S$ and $T$, represented by dual pairs $\set{D_1, D_2}$ and $\set{E_1, E_2}$, respectively, there is a Goeritz group element, $g_1$ taking $D_1$ to $E_1$. Now, $g_1(D_2) = E_3$, for some disk $E_3$ dual to $E_1$. Then, there is a Goeritz group element, $g_2$, in the stabilizer of $E_1$, taking $E_3$ to $E_2$. Hence, $(g_2 \circ g_1)$ carries $\set{D_1, D_2}$ to $\set{E_1, E_2}$, and hence $S$ to $T$. Thus, $\langle \langle \beta^2 \rangle \rangle$ is the group of all reducing sphere twists, contained within the Goeritz group.

Conversely, if $K$ and $K'$ are related by a Goeritz group element, as they are homologous, they must be related by $\ker(q \circ *) = \langle \langle \beta^2 \rangle \rangle$, by Theorem \ref{thm:kernel}.

Finally, the curve $K$ represents elements of each of the two free groups $\pi_1(V)$ and $\pi_1(W)$, up to free homotopy, by pushing $K$ slightly off of $F$ into either of $V$ or $W$. Evidently, a reducing sphere twist of $K$ will not change these free homotopy classes.
\end{proof}

\section{Acknowledgments}

The authors would like to thank Neil Hoffman for asking a question that led to this line of inquiry, W. Riley Casper for insight offered in proving a tricky case, and the referee for their helpful comments. The first author was partially supported by the AMS-Simons Research Enhancement Grant for Primarily Undergraduate Institution Faculty.

\bibliographystyle{gtart}
\bibliography{GoeritzEquivalenceLensSpaces}

\end{document}